\documentclass[a4paper,reqno, 11pt]{amsart}
\usepackage{graphicx,color}
\usepackage{hyperref}
\usepackage{footmisc}
\usepackage{amsmath}
\usepackage{tikzsymbols}
\usepackage{amsfonts, dsfont}
\usepackage{amssymb}
\usepackage{bbm}
\usepackage{epstopdf}
\usepackage{color}

\usepackage[utf8]{inputenc}
\graphicspath{ {images/} }
\usepackage[a4paper]{geometry}
\usepackage{tikz}
\usepackage{amsthm}
\usepackage{mathtools}
\usepackage{float}

\DeclareMathOperator{\adj}{adj}

\def\openone{{\mathbbm{1}}}
\numberwithin{equation}{section}

\usepackage{graphicx}

\title{Weakly non-planar dimers}

\author{Alessandro Giuliani}
\address{Dipartimento di Matematica e Fisica, Università di Roma Tre, L.go S. L.
Murialdo 1, 00146 Roma, Italy} \email{alessandro.giuliani@uniroma3.it}
\author{Bruno Renzi}
\address{Dipartimento di Matematica e Fisica, Università di Roma Tre, L.go S. L.
Murialdo 1, 00146 Roma, Italy} \email{bruno.renzi@uniroma3.it}
\author{Fabio Lucio Toninelli}
\address{Technical University of Vienna, Institut f\"ur Stochastik und Wirtschaftsmathematik, Wiedner Hauptstra{\ss}e 8-10, A-1040 Vienna, Austria} \email{fabio.toninelli@tuwien.ac.at}

\newtheorem{theorem}{Theorem}

\newtheorem{remark}{Remark}

\newtheorem{assumption}{Assumption}

\newtheorem{lemma}{Lemma}
\newtheorem{definition}{Definition}
\newtheorem{proposition}{Proposition}

\newcommand{\rvline}{\hspace*{-\arraycolsep}\vline\hspace*{-\arraycolsep}}
\newcommand{\un}[1]{\underline{#1}}

     \let\th=\theta  \let\l=\lambda

   \let\o=\omega

\begin{document}

\maketitle
\begin{abstract}
We study a model of fully-packed dimer configurations (or perfect
matchings) on a bipartite periodic graph that is two-dimensional but
not planar. The graph is obtained from $\mathbb Z^2$ via the addition
of an extensive number of extra edges that break planarity (but not bipartiteness). We prove
that, if the weight $\lambda$ of the non-planar edges is small enough,
{a suitably defined} height function scales on large distances to the Gaussian Free
Field with a $\lambda$-dependent amplitude, that coincides with the anomalous exponent of dimer-dimer correlations. Because of
non-planarity, Kasteleyn's theory does not apply: the model is
{not integrable}. Rather, we map the model to a system of interacting
lattice fermions in the Luttinger universality class, which we then analyze via fermionic Renormalization Group methods.
\end{abstract}

\section{Introduction}
The understanding of the rough phase of two-dimensional random
interfaces is important in connection with the macroscopic fluctuation
properties of equilibrium crystal shapes and of the separation surface
between coexisting thermodynamic phases.
A classical instance  arises when studying the low
temperature properties of the three-dimensional (3D) Ising model in
the presence of Dobrushin Boundary Conditions (DBC). If DBC are fixed
so to induce a horizontal interface between the $+$ and $-$ phases, it
is well known \cite{Dobrushin} that at low enough temperatures the
interface is rigid. It is conjectured that between the so-called
roughening temperature and the Curie temperature, the interface
displays fluctuations with unbounded variance (the variance diverges
logarithmically with the system size), and the height profile
supposedly has a massless Gaussian Free Field (GFF) behavior at large
scales. This conjecture is completely open, in fact not even the
existence of the roughening temperature has been proved. A connected
result \cite{FS} is  logarithmic divergence of fluctuations of the 2D
SOS interface at large enough temperature; however, the result comes
with no control of the scaling limit. If DBC are chosen
so to induce a `tilted' interface, say orthogonal to the $(1,1,1)$
direction, then things are different: fluctuations of the interface
are logarithmic already at zero temperature; an exact mapping of the
height profile and of its distribution into the dimer model
on the hexagonal lattice, endowed with the uniform measure, allows one to get a full control on the
large scale properties of the interface fluctuations, which are now
proved to behave like a GFF (see \cite{KOS} for the covariance
structure, and \cite{Kenyon_notes}, as well as \cite[Section
3]{GMTaihp} for the full Gaussian limit). It is very likely that the
GFF behavior survives the presence of a small but positive
temperature; however, the techniques underlying the proof at zero
temperature, based on the exact solvability of the planar dimer
problem, break down. At positive temperatures, the very
notion of height of the interface is not  well defined,
because of overhangs; these will have a low but non-zero density at
low temperature. {It is likely, though, that the  height, even if not defined
everywhere at a microscopic level, may be well-defined in a coarse grained sense; therefore, one can still 
ask about the large-scale  behavior of its fluctuations. The coarse-grained height should admit an effective description in terms of 
a dimer model, whose distribution, however, is not expected to be uniform:} temperature induces an effective `interaction' among
dimers. 

In a previous series of works
\cite{GMTaihp,GMTHaldane,frattaglie,GMT20}, in collaboration with
V. Mastropietro, we started developing methods for the treatment of
 non-solvable dimer models via constructive, fermionic,
Renormalization Group (RG) techniques.  We exhibited an explicit class
of models, which include the 6-vertex model close to its free Fermi
point as well as several non-integrable versions thereof, for which we
proved scaling to the  GFF,
as well as the validity of a `Kadanoff' or `Haldane' scaling relation
connecting the critical exponent of the so-called electric correlator
with the one of the dimer-dimer correlation. Such a scaling relation
is the counterpart, away from the free Fermi point, of the
universality of the stiffness coefficient of the GFF first observed in
\cite{KOS}, in connection with the fact that the spectral curve of a
planar bipartite dimer model is a Harnack curve.

In this paper, {motivated by our wish to understand the height fluctuations in situations where 
the height function is not locally well-defined at a microscopic level but only in a coarse-grained sense, as in the case of the 3D Ising interface discussed above, 
and in situations where the planarity assumption on the underlying  graph  fails to be satisfied\footnote{{The interacting dimer model with plaquette interaction studied in \cite{Alet} and in \cite{HP}, which motivated our series of works \cite{GMTaihp,GMTHaldane,frattaglie,GMT20}, 
is a toy model for short range Resonance Valence Bond ground states and for liquid crystals in two dimensions. As clear from its definition, such a model 
is based on drastic simplifications of the physical phenomena one intends to study. In particular, the planarity assumption is physically unjustified: in realistic situations, nothing 
prevents the presence of defects allowing the dimers 
to arrange on a bond connecting pairs of sites beyond the nearest neighbors.}}},  we generalize our analysis to a new setting, inspired by
a problem proposed by S. Sheffield a few years ago\footnote{Open problem session at the workshop: ``Dimers, Ising Model, and their Interactions'', BIRS, 2019}. Namely, we study the
large scale properties of a suitably defined height function, for a
dimer model that is \emph{two-dimensional but non-planar}. In short, we introduce a
`weakly non-planar' dimer model, by adding non-planar edges with small
weights to a reference planar square lattice. We do so in a periodic
fashion, and in such a way that non-planar edges are restricted to
belong to cells, separated among each other by corridors
of width one, which are crossed by none of the non-planar edges. The
fact that non-planar edges avoid these corridors allows us to define a
notion of height function on the faces belonging to the corridors
themselves. We prove that this height function scales at large
distances to a GFF with stiffness coefficient that is \textit{equal}
to the anomalous critical exponent of the two-point dimer-dimer
correlation.  

As in \cite{GMTaihp,GMT20}, the proof is based on an exact
representation of the dimer model as a system of interacting lattice
fermions and in a rigorous multiscale analysis of the effective
fermionic model, which has the structure of a lattice regularization
of a Luttinger-type model. With respect to the previous works
\cite{GMTaihp,GMT20}, obtaining a fermionic representation turns out
to be much less trivial, due to the loss of planarity.  The infrared
(i.e., large-scale) analysis of the lattice fermionic model is
performed thanks to a comparison with a solvable reference continuum
fermionic model, which has been studied and constructed in a series of
works by G. Benfatto and V. Mastropietro \cite{BFM09a,BFM09b,BM02,BM04,BM05,BM10,BM11}, partly in collaboration also with P. Falco \cite{BFM09a,BFM09b,BFM14}.
The GFF behavior and the
Kadanoff-Haldane scaling relation of the dimer model follow from 
a careful comparison between the emergent chiral Ward Identities of the
reference model with exact lattice Ward Identities of the dimer model.

The first novelty of the present work, as compared to
\cite{GMTaihp,GMTHaldane,GMT20}, is related to the fermionic
representation of the weakly non-planar model. The presence of
non-planar edges requires a quite non-trivial adaptation of
Kasteleyn's theory, which is needed for the very formulation of the
finite-volume model in terms of a non-Gaussian Grassmann integral. In
fact, our non-planar model can in general be embedded on a surface of
minimal genus $g\approx L^2 $ (of the order of the number of
non-planar edges) and Kasteleyn's theory for the dimer model on
general surfaces \cite{Galluccio,Tesler} would express its partition
function as the sum of $4^g$ determinants, i.e. of $4^g$ Gaussian
Grassmann integrals, a rewriting that is not very useful for
extracting thermodynamic properties.  In this respect, the remarkable
aspect of Proposition \ref{prop:1} below is that it expresses the
partition function of \emph{just four Grassmann integrals, which are,
  however, non-Gaussian}.  The second novel ingredient of our
construction is the identification {(via the block-diagonalization
  procedure of Section \ref{sec:5.2})} of massive modes associated
with the Grassmann field which enters the fermionic representation of
the model. The fact that the elementary cell of our model consists of
$m^2$ sites, with $m$ an even integer larger or equal to $4$, implies
that the basic Grassmann field of our effective model has a minimum of
16 components. It is well known \cite{Gia,Tsv} that multi-component
Luttinger models, such as the 1D Hubbard model \cite{LiWu}, to cite
the simplest possible example, do not necessarily display the same
qualitative large distance features as the single-component one: new
phenomena and quantum instabilities, such as spin-charge separation
and metal-insulator transitions accompanied with the opening of a Mott
gap may be present and may drastically change the resulting
picture. Therefore, it is a priori unclear whether the height function
of our model should still display a GFF behavior at large
scales. Remarkably, however, the fact that the characteristic
polynomial of the reference model has at most two simple zeros,
{as proven in \cite{KOS}, directly implies} that all but two of the
components of the effective Grassmann field are massive, and they can
be preliminarily integrated out. This way, one can at last re-express
the effective massless model in terms of just two massless fields
(quasi-particle fields), in a way suitable for the application of the
multi-scale analysis developed in \cite{GMTaihp,GMT20}.  At this
point, a large part of the multi-scale analysis is based on the tools
developed in our previous works, which we will refer to for many
technical aspects, without repeating the analysis in the present
slightly different setting.

{As a side remark, we emphasize that the massive modes would arise
  also in the ``planar interacting dimer models'' of
  \cite{GMTaihp,GMT20}, if one worked on a graph whose fundamental cell
  contains $\ell\ge 2$ black/white vertices (an example is the square-octagon
  graph, see Fig. 5 in \cite{KOS}, where $\ell=4$). In contrast, one has
  $\ell=1$ in the context of \cite{GMTaihp,GMT20}.} {Our procedure consisting of (i) integrating out the massive 
 degrees of freedom and (ii) reducing to an effective massless model, implies in particular that 
the results proved in \cite{GMTaihp,GMT20} for planar interacting dimer models on the square lattice extend to the case
of general $\mathbb Z^2$-periodic two-dimensional bipartite lattices.}

\subsection{The broader context: height delocalization for discrete interface models}

Recently, remarkable progress has been made on (logarithmic)
delocalization of discrete, two-dimensional random interfaces.  We
start with the result which is maybe the closest in spirit to our
work, that is \cite{Roland1,Roland2}. These works prove, by means of bosonic,
constructive RG methods, that the height function of the discrete
Gaussian interface model (that is the lattice GFF conditioned to
be integer-valued) has, at sufficiently high temperature, the
continuum GFF as scaling limit. In a way, this result is quite complementary to ours,
since the model considered there is a perturbation of a \emph{free bosonic}
model (the lattice GFF), while in our case we perturb around a \emph{free
fermionic} one (the non-interacting dimer model). For closely related results on the 2D lattice Coulomb gas, see also \cite{Fa1,Fa2}.

In a broader perspective, there has been a number of recent results
(e.g. \cite{AHPS,CPST,Lammers1}) that prove delocalization of discrete,
two-dimensional interface models at high temperature, even though they
fall short of proving convergence to the GFF. Let us mention in
particular the recent \cite{Lammers1}, which proves with a rather soft argument a
(non-quantitative) delocalization statement for rather general height
models, under the restriction, however, that the underlying graph has maximal
degree three. For the particular case of the 6-vertex model, delocalization of the height function is known to hold in several regions of parameters \cite{DC1,DC2,Lis,Wu}
but full scaling to the GFF has been proven only in a neighborhood of the free fermion point \cite{GMT20}.

\subsection*{Organization of the article}

The rest of the paper is organized as follows: in Section \ref{sec:2}
we define the model and state our main results precisely. In Section
\ref{sec:2} we review some useful aspects of Kasteleyn's theory on toroidal graphs and
derive the Grassmann representation of the weakly non-planar dimer
model. In Section \ref{sec:4} we prove one of the main results of our
work, concerning the logarithmic behavior of the height covariance at
large distances and the Kadanoff-Haldane scaling relation, assuming
temporarily a sharp asymptotic result on the correlation functions of
the dimer model. The proof of the latter is based on a generalization
of the analysis of \cite{GMT20}, described in Section
\ref{lasezionepesante}. As mentioned above, the novel aspect of this
part consists in the identification and integration of the massive
degrees of freedom (Sections \ref{sec:5.1}-\ref{sec:5.3}), while the integration of the massless ones (Section \ref{sec:5.4})  is
completely analogous to the one described in \cite{GMT20}. Finally,
in Section \ref{sec:6}, we complete the proof of the convergence of
the height function to the GFF.
\section{Model and results}\label{sec:2}

\subsection{The ``weakly non-planar'' dimer model}

\label{sec:wnpdm}

To construct the graph $G_L$ on which our dimer model is defined, we let $L,m$ be two positive integers with {$m\ge4$} even, and we 
start with $G^{0}_L=(\mathbb Z/(L m\,\mathbb Z))^2$, which is just the
toroidal graph obtained by a periodization of $\mathbb Z^2$ with
period $L\, m$ in both horizontal and vertical directions. We
can partition $G^0_L$ into $L^2$ square cells $B_{x}$, $x=(x_1,x_2)\in \Lambda:=(-L/2,L/2]^2\cap\mathbb Z^2$, of
sidelength $m$. The graph $G^0_L$ is plainly bipartite and we color vertices of
the two sub-lattices black and white (each cell contains $m^2/2$
vertices of each color).  Black (resp. white) vertices are denoted $b$
(resp. $w$).  We let $ {\bf e_1}$ (resp. ${\bf e_2}$) denote the
horizontal (resp. vertical) vectors of length $m$ and we note that
translation by  ${\bf e_{1}}$ (resp. by ${\bf e_2}$) maps the cell $B_x$ into $B_{((x_1+1) \mod L,x_2)}$ 
(resp. $B_{( x_1, (x_2+1)\mod L)}$).  A natural choice of coordinates for vertices is the following
one: a vertex is identified by its color (black or white) and by a
pair of coordinates $(x,\ell)$ where
$x$ identifies the label of
the cell the vertex belongs to, and the ``type'' $\ell\in\mathcal I:=\{1,\dots,m^2/2\}$
identifies the vertex within the cell. It does not matter how we label
vertices within a cell, but we make the natural choice that if two
vertices are related by a translation by a multiple of
${\bf e_{1},e_2}$, then they have the same type index $\ell$.

The graph $G_L$ is obtained from $G^{0}_L$ by adding in each cell
$B_x$ a finite number of edges among vertices of opposite color (so
that $G_L$ is still bipartite), with the constraint that $G_L$ is
invariant under translations by multiples of ${\bf e_{1},e_2}$ (i.e.,
vertex $w$ of coordinates $(x,\ell)$ is joined to vertex $b$ of
coordinates $(x,\ell')$ if and only if the same holds for any other
$x'\in\Lambda$). 
See Fig. \ref{fig:1} for an example.
\begin{remark}
  It is easy to see that we need that $m\ge4$ for this construction to work: if $m=2$, the two black edges in the cell are already connected to the two black vertices and there are no non-planar edges that can be added.
\end{remark}
Note that $G_L$ is in
general non-planar, even in the full-plane limit $L\to\infty$. 
\begin{figure}
\centering
\includegraphics[scale=1]{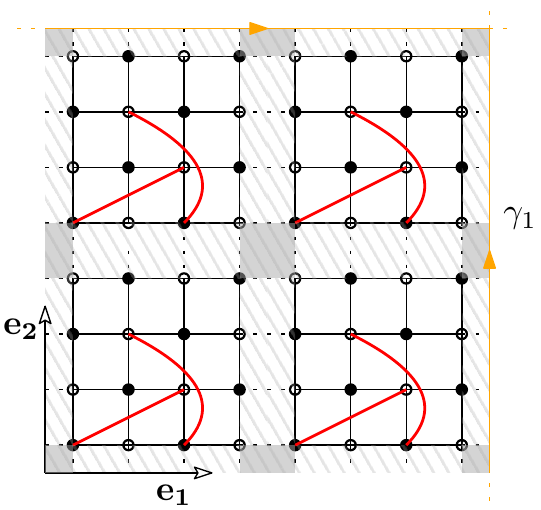}
\caption{An example of graph $G_L$ with $L=2$, $m=4$ and two
  non-planar edges (in red) per cell. The height function is defined
  on the set $F$ of dashed faces outside of cells. Faces colored gray
  are those in $\bar F$ (they share a vertex with four different
  cells).}\label{fig:1}
\end{figure}

Let $E_L$ denote the set of edges of $G_L$: we write $E_L$ as the disjoint union $E_L=E^0_L\cup N_L$ where 
$E^0_L$ are the edges  of $G^{0}_L$ (we call these ``planar edges'') and $N_L$ 
 (we call these ``non-planar edges'') are the extra ones. Each edge  $e\in E_L$ 
is assigned a positive weight: since we are interested in the
situation where the weights of non-planar edges are small compared to
those of planar edges, we first take a collection of weights $\{\tilde
t_e\}_{e\in E_L}, \tilde t_e>0$ that is invariant under translations by multiples of ${\bf
  e_{1},e_2}$, and then we establish that the weight of a planar edge
$e$ is $t_e=\tilde t_e$ while that of a non-planar edge is
$t_e=\lambda \tilde t_e$, where $\lambda$ is a real parameter,
that will be taken small later. 

To simplify expressions that follow, we will sometimes write ${\bf x}
\in {\bf \Lambda}$ instead of $(x, \ell) \in \Lambda \times
\mathcal{I}$ for the coordinate of a vertex of $G_L$.  Also, we label
the collection of edges in $E_L$ whose black vertex has type
$\ell$ with a label $j\in \mathcal J_{\ell}=\{1,\dots,|\mathcal
J_{\ell}|\}$. The labeling is done in such a way that two edges that
are obtained one from the other via a translation by a multiple of
${\bf e_1,e_2}$ have the same label.  Note that $|\mathcal J_{\ell}|\ge4$, and it is strictly larger than four if there are
non-planar edges incident to the black vertex of type $\ell$. By convention, we label
$j=1,\dots 4$ the four edges of $G^0_L$ belonging to $\mathcal J_{\ell}$, starting from the horizontal one whose left endpoint is black,
and moving anti-clockwise.

The set of perfect matchings (or dimer configurations) of $G_L$ is
denoted $\Omega_L$. Each $M\in \Omega_L$ is a subset of $E_L$ 
and the set of perfect matchings that contain only planar edges is
denoted $\Omega^{0}_L$.  Our main object of study is the probability
measure on $\Omega_L$ given by
\begin{equation}
\label{eq:ZLl}
\mathbb{P}_{L,\lambda}(M)=\frac{w(M) }{Z_{L,{\bf \lambda}}}\openone_{ M\in\Omega_L}, \;
w(M)=\prod_{e \in M} t_e,\;  Z_{L,\lambda}=\sum_{M\in\Omega_L} w(M).
\end{equation}

We are interested in the limit where $L$ tends to infinity while $m$
(the cell size) is fixed. In this limit, the graph $G^{0}_L$ becomes
the (planar) graph $G^{0}_\infty=\mathbb Z^2$ while $G_L$ becomes a periodic,
bipartite but in general non-planar graph $G_\infty$. Cells $B_x$ of the
infinite graphs are labelled by $x\in \mathbb Z^2$. We let $\Omega$
(resp. $\Omega^{0}$) denote the set of perfect matchings of $G_\infty$
(resp. of $\mathbb Z^2$).

In the case $\lambda=0$, the measure $\mathbb{P}_{L,0}$ is supported
on $\Omega^{0}_L$: in fact, $\mathbb P_{L,0}$ is just the
Boltzmann-Gibbs measure of the dimer model on the (periodized) square
grid, with edge weights of periodicity $m$ (we will refer to this as
the ``non-interacting dimer model''). The non-interacting model is
well understood via Kasteleyn's \cite{Kasteleyn,Kasteleyn2,Kenyonlocal} (and
Temperley-Fisher's \cite{TemperleyFisher}) theory, that allows to
write its partition and correlation functions in determinantal form.
According to the choice of the edge weights $\{t_e\}$, the
non-interacting model can be either in a liquid (massless), gaseous
(massive) or frozen phase, see \cite{KOS}. In particular, in the
liquid phase correlations decay like the squared inverse distance (see
for instance \eqref{asymcorrfunc}-\eqref{asym0} for a more precise
statement).  \emph{In this work, we assume that the edge weights  are such that for $\lambda=0$, the model is in a massless phase.}

The essential facts from Kasteleyn's theory that are needed for the
present work are recalled in Section \ref{sec:KastTheory}.  In
particular, we emphasize that all the statistical properties of the
non-interacting model are encoded in the so-called
\emph{characteristic polynomial} $\mu$ (see \eqref{eq:mu}), that is
nothing else but the determinant of the Fourier transform of the
so-called Kasteleyn matrix. Then, the assumption that the $\lambda=0$
model is in the massless phase, can be more precisely stated as
follows:
\begin{assumption}
 \label{ass:1}
  The edge weights $\{t_e\}$ are such that the {``characteristic
    polynomial'' $\mu:[-\pi,\pi]^2\mapsto \mathbb C$} (see formula \eqref{eq:mu} below and the discussion following it) of the non-interacting dimer model has exactly two zeros {$p^+_0,p^-_0$ (distinct and simple)}. 
\end{assumption}

We recall from \cite{KOS} that this is a non-empty condition on the
edge weights (in fact, this set of edge weights is a non-trivial open
set). {{We also remark that if Assumption \ref{ass:1} is not satisfied, then we are in one of the following situations: 
    \begin{enumerate}
    \item  The edge weights are such that $\mu$ has no zeros 
    on $[-\pi,\pi]^2$, corresponding to the frozen or to the gaseous
    phases of the non-interacting dimer model. In this case, the dimer model can be easily shown to
    be stable under the addition of dimer interaction such as those
    treated in \cite{GMT20} or of non-planar edges of small weight
    such as those treated in this paper. By `easily', we mean `via a
    single-step fermionic cluster expansion', which shows
    that the fluctuations of the perturbed model have the same
    qualitative behavior of the unperturbed ones. 
    In this case, the height function displays no
    interesting behavior in the scaling limit.    
  \item The edge weights are such that $\mu$ has one double zero $p_0$ and: either the system is at the boundary separating the liquid from the gaseous or frozen phases (in which case $\nabla\mu(p_0)\neq0$), or it is at a degenerate point within the liquid phase
 (in which case, called `real node' in \cite{KOS}, $\nabla\mu(p_0)=0$). 
 These cases display rich and interesting behaviors: e.g., the fluctuations of the height level lines of the integrable dimer model, close to the liquid/frozen and liquid/gaseous boundaries converge to the so-called Airy process \cite{Beffara-Johansson, Johansson}. We postpone the analysis of non-integrable dimer models in such critical cases to future work.
    \end{enumerate}}}
\subsection{Results}
\label{sec:results}

Our main goal is to understand the large-scale properties of the height function under the limit measure $\mathbb P_\lambda$, which is the weak limit as  $L\to\infty$ of $\mathbb
  P_{L,\lambda}$. The fact that this limit exists, provided that $|\lambda|\le \lambda_0$ for a sufficiently small $\lambda_0$, is a byproduct of the proof.

  Our first main result concerns the  large distance asymptotics of the truncated dimer-dimer correlations. We use the notation $\mathds 1_e$ for the indicator function that the edge $e$ is occupied by a dimer, and $\mathbb E_\lambda(f;g)$ for $\mathbb E_\lambda(f g)-\mathbb E_\lambda(f)\mathbb E_\lambda(g)$.

\begin{theorem}\label{thm:1} Choose the dimer weights on the planar edges as in Assumption \ref{ass:1}.  
  There exists   $\lambda_0>0$ and analytic functions $\nu:[-\lambda_0,\lambda_0]\mapsto \mathbb R^+$,
  $\alpha_\omega,\beta_\omega:[-\lambda_0,\lambda_0]\mapsto \mathbb C\setminus\{0\}$ (labelled by $\omega\in\{+,-\}$ and satisfying $\overline{\alpha_+}=- \alpha_{-},  \overline{\beta_+}=- \beta_{-}$ and $\alpha_\omega(\lambda)/\beta_\omega(\lambda)\not\in\mathbb R$), 
  $K_{\omega,j,\ell}, H_{\omega,j,\ell}:[-\lambda_0,\lambda_0]\mapsto \mathbb C$ (labelled by $\omega\in\{+,-\}$, $\ell\in\mathcal I$, $j\in\mathcal J_{\ell}$ and satisfying 
  $K_{\omega,j,\ell}=\overline{K_{-\omega,j,\ell}}$, $H_{\omega,j,\ell}=\overline{H_{-\omega,j,\ell}}$) and $p^\omega:[-\lambda_0,\lambda_0]\mapsto [-\pi,\pi]^2$ (labelled by $\o\in\{-1,+1\}$ and satisfying $p^+=-p^-$) 
   such that, for any two edges $e,e'$ 
  with black vertices $(x,\ell), (x',\ell')\in \mathbb Z^2\times\mathcal I$ such that $x\neq x'$ and labels $j\in\mathcal J_{\ell}$, $j'\in\mathcal J_{\ell'}$, 
\begin{equation}\label{eq:asy}
\mathbb E_{\lambda}(\mathds 1_{e};\mathds 1_{e'})=\sum_{\omega}\Biggl[\frac{K_{\omega,j,\ell}K_{\omega,j',\ell'}}{(\phi_\omega(x-x'))^2}+
\frac{H_{\omega,j,\ell}H_{-\omega,j',\ell'}}{|\phi_\omega(x-x')|^{2\nu}}e^{2ip^\omega\cdot(x-x')}\Biggr]+\text{Err}(e,e'),\end{equation}
where, letting $x=(x_1,x_2)$,
\begin{eqnarray}
  \label{eq:phi}
  \phi_\omega(x):=\omega(\beta_\omega(\lambda)x_1-\alpha_\omega(\lambda)x_2)
\end{eqnarray}
and $|\text{Err}(e,e')|\le C|x-x'|^{-3+O(\lambda)}$ for some $C>0$ and $O(\lambda)$ independent of $x,x'$. 
Moreover, $\nu(\l)=1+O(\l)$.
\end{theorem}

Even if not indicated explicitly, the functions $\nu,\alpha_\omega,\beta_\omega, K_{\omega,j,\ell}, H_{\omega,j,\ell},p^\o$ all depend non-trivially on the
edge weights $\{t_e\}$. In particular, generically, $\nu=1+c_1\lambda+O(\lambda^2)$, with $c_1$ a non zero coefficient, which depends upon the edge weights (this was already observed in \cite{frattaglie}
for interacting dimers on planar graphs); therefore, generically, $\nu$ is larger or smaller than $1$, depending on the sign of $\lambda$. 

At $\lambda=0$, \eqref{eq:asy} reduces to the known asymptotic formula
for the truncated dimer-dimer correlation of the standard planar dimer
model, which is reviewed in the next section. The most striking
difference between the case $\lambda\neq0$ and $\lambda=0$ is the
presence of the critical exponent $\nu$ in the second term in square
brackets in the right hand side of \eqref{eq:asy}. It shows that the
presence of non-planar edges in the model qualitatively changes the
large distance decay properties of the dimer-dimer
correlations. Therefore, naive universality, in the strong sense that
all critical exponents of the perturbed model are the same as those of
the reference unperturbed one, fails.  In the present context the
correct notion to be used is that of `weak universality', due to
Kadanoff, on the basis of which we expect that the perturbed model is
characterized by a number of exact scaling relations; these should
allow us to reduce all the non-trivial critical exponents of the
model (i.e., those depending continuously on the strength of the
perturbation) to just one of them, for instance $\nu$ itself. A
rigorous instance of such a scaling is discussed in Remark
\ref{rem:striking} below.

The weak universality picture is formally predicted by bosonization methods (see the introduction of \cite{GMTHaldane} for a brief overview), 
which allow one to express the large distance asymptotics of all correlation functions in terms of a single, underlying, massless GFF. Such a GFF is nothing but the scaling limit of the height function of the model, as discussed in the following. Given a perfect matching $M\in\Omega^{0}$ of the infinite graph
$G^{0}_\infty=\mathbb Z^2$, there is a standard definition of height
function on the dual graph: given two faces $\xi,\eta$ of $\mathbb
Z^2$, one defines
\begin{align}\label{heightfunction} h(\eta)-h(\xi)=\sum_{e \in C_{\xi \to \eta}}\sigma_e\Big{(}\mathbbm{1}_{e \in M}-\frac14 \Big{)}
\end{align}
together with $h(\xi_0)=0$ at some reference face $\xi_0$. Here, $C_{\xi \to \eta}$ is a nearest-neighbor path from $\xi$ to $\eta$ and $\sigma_e$ is a sign which equals $+1$ if the edge $e$ is crossed with the white vertex on right and $-1$ otherwise). The definition is well-posed since it is independent of the choice of the path. We recall that under $\mathbb P_0$, the height 
function is known to admit a  GFF scaling limit \cite{Kenyon_notes,GMTaihp}.

A priori, on a non-planar graph such as $G$, there is no canonical
bijection between perfect matchings and height functions. However, since the
non-planarity is ``local'' (non-planar edges do not connect different
cells), there is an easy way out. Namely, let $F$ denote the set of
faces of $\mathbb Z^2$ that do not belong to any of the cells $B_x$
(see Figure \ref{fig:1}). Given a perfect matching $M\in \Omega$, define an
integer-valued height function $h$ on faces $\xi\in F$ by setting it
to zero at some reference face $\xi_0\in F$ and by imposing
\eqref{heightfunction} for any path $C_{\xi\to\eta}$ that uses only
faces in $F$. It is easy to check that $h$ is then
independent of the choice of path.

Our second main result implies in particular that the variance of the height
difference between faraway faces in $F$ grows logarithmically with the
distance. For simplicity, let us restrict our attention to the subset
$\bar F\subset F$ of faces that share a vertex with four cells (see Fig. \ref{fig:1}): if a face in $ \bar F$ shares a vertex with 
$B_x,B_{x-(0,1)}, B_{x-(1,0)},B_{x-(1,1)}$, then we denote it  by $\eta_x$. 
\begin{theorem}\label{mainthrm} Under the same assumptions as Theorem \ref{thm:1}, 
for $x^{(1)},\dots, x^{(4)}\in \mathbb Z^2$,
\begin{multline}
\label{eq:35l}
\mathbb E_\l\left[(h(\eta_{x^{(1)}})-h(\eta_{x^{(2)}}));(h(\eta_{x^{(3)}})-h(\eta_{x^{(4)}}))
\right]\\=\frac {\nu(\lambda)}{2\pi^2}
    \Re\left[\log \frac{( \phi_+({x^{(4)}})- \phi_+({x^{(1)}}))( \phi_+({x^{(3)}}) - \phi_+({x^{(2)}}))}{(\phi_+({x^{(4)}})- \phi_+({x^{(2)}}))(\phi_+({x^{(3)}})-\phi_+({x^{(1)}}))}
    \right]\\
+O\left(\frac1{\min_{i\ne j\le 4}|x^{(i)}-x^{(j)}|^{1/2}+1}\right)
\end{multline}
where $\nu$ and $\phi_+$ are the same as in Theorem \ref{thm:1}
\end{theorem}
Note that in particular, taking $x^{(1)}=x^{(3)}=x,x^{(2)}=x^{(4)}=y$ we have
\begin{eqnarray}
  {\rm Var}_{\mathbb P_\lambda}(h(\eta_{x})-h(\eta_{y}))= \frac{\nu(\lambda)}{\pi^2}\Re \log(\phi_+(x)-\phi_x(y))+O(1)=\frac{\nu(\lambda)}{\pi^2}\log|x-y|+O(1)
\end{eqnarray}
as $|x-y|\to\infty$.
\begin{remark}
  \label{rem:striking}
The remarkable fact of this result is that the `stiffness' coefficient $\nu(\lambda)/\pi^2$ of the GFF is the same, up to the $1/\pi^2$ factor, as the 
critical exponent of the oscillating part of the dimer-dimer correlation. There is no a priori reason that the two coefficients should be the same, and it is actually a deep implication of our proof that this is the case. Such an identity is precisely one of the scaling relations predicted by Kadanoff and Haldane in the context of the {{8-}}vertex, Ashkin-Teller, XXZ, and Luttinger liquid models, which are different models in the same universality class as our non-planar dimers (see \cite{GMTHaldane} for additional discussion and references).
\end{remark}

Building  upon the proof of Theorem \ref{mainthrm}, we also obtain bounds on the higher point cumulants of the height; these, in turn, imply convergence of the height profile 
to a massless GFF: 
\begin{theorem}
  \label{th:GFF}
  Assume that $|\lambda|\le \lambda_0$, with $\lambda_0>0$ as in Theorem \ref{mainthrm}.  For every $C^\infty$, compactly supported test function $f:\mathbb R^2\mapsto\mathbb R$ of zero average and $\epsilon>0$, define
  \begin{equation}
    \label{eq:htestata}
    h^\epsilon(f):=\epsilon^2 \sum_{x\in\mathbb Z^2} \big(h(\eta_x)-\mathbb E_\lambda(h(\eta_x))\big)  
    f(\epsilon x).
  \end{equation} 
  Then, one has the convergence in distribution
  \begin{equation}
    \label{eq:convGFF}
    h^\epsilon(f)\stackrel{\epsilon\to0}
   \Longrightarrow \mathcal N(0,\sigma^2(f))
  \end{equation}
  where $\mathcal N(0,\sigma^2(f))$ denotes a centered Gaussian distribution of variance
  \[
\sigma^2(f):=\frac{\nu(\lambda)}{2\pi^2}\int_{\mathbb R^2}dx\int_{\mathbb R^2}dy f(x)\Re[\log \phi_+(x-y)]f(y).
  \]
\end{theorem}

\section{Grassmann representation of the generating function}\label{sec:3}

In this section we rewrite the partition function $Z_{L,\lambda}$ of
\eqref{eq:ZLl} in terms of Grassmann integrals (see Sect.\ref{sec:3.3}). As a byproduct of our construction, we obtain a similar Grassmann representation 
for the generating function of correlations of the dimer model. We also observe that the Grassmann integral for the generating function is invariant under 
a lattice gauge symmetry, whose origin has to be traced back to the local conservation of the number of incident dimers at lattice sites, and which 
implies exact lattice Ward Identities for the dimer correlations (see Sect.\ref{sec:3.4}). 

Before diving into the proof of the Grassmann representation, it is convenient to recall some preliminaries 
about the planar dimer model, its Gaussian Grassmann representation and the structure of its correlation functions  in the thermodynamic limit. This will be done in the next two subsection, Sect.\ref{sec:KastTheory}, \ref{sec:KG0}

\subsection{A brief reminder of Kasteleyn theory}
\label{sec:KastTheory}

Here we recall a few basic facts of Kasteleyn theory for the dimer
model on a bipartite graph $G=(V,E)$ embedded on the torus, with edge
weights $\{t_e>0\}_{e\in E}$. For later purposes, we need this for
more general such graphs than just $G^0_L$.  For details we refer
to \cite{Galluccio}, which considers the more general case where the
graph is not bipartite and it is embedded on an orientable surface of
genus $g\ge 1$. For the considerations of this section, we do not need the  edge weights to display any periodicity, so here we will work 
with generic, not necessarily periodic, edge weights. 

As in \cite{Galluccio}, we assume that $G$ can be represented as a
planar connected graph $G_0=(V,E_0)$ (we call this the ``basis graph of $G$''), embedded on a square, with
additional edges that connect the two vertical sides of the square
(edges $E_1$) or the two horizontal sides (edges $E_2$). Note that
$E=E_0\cup E_1\cup E_2$. See Figure \ref{fig:6}.  We always assume that the basis graph $G_0$ is
connected and actually\footnote{\cite{Galluccio} develops Kasteleyn's
  theory without assuming that $G_0$ is $2-$connected. We will avoid
  below having to deal with non-$2$-connected graphs, which would
  entail several useless complications} that it is $2-$connected (i.e. removal
of any single vertex together with the edges attached to it does not
make $G_0$ disconnected). We also assume that $G_0$ admits at least one perfect matching and we fix a reference one, which we call $M_0$.

Following the terminology of \cite{Galluccio}, we introduce the following definition.
\begin{definition}[Basic orientation]
  \label{def:bo}
We call an orientation
$D_0$ of the edges $E_0$ a ``basic orientation of $G_0$'' if all the
internal faces of the basis graph $G_0$ are clock-wise odd, i.e. if running clockwise
along the boundary of the face, the number of co-oriented edges is odd
(since $G_0$ is 2-connected, the boundary of each face is a
cycle).  
\end{definition}
 A basic orientation always exists \cite{Galluccio}, but 
in general it is not unique. Next, one defines $4$ orientations of the full
graph $G$ as follows (these are called ``relevant orientations'' in
\cite{Galluccio}). First, one draws the planar graphs $G_j,j=1,2$
whose edge sets are $E_0\cup E_j$, as in Fig. \ref{fig:6}.
\begin{figure}[H]
\centering
\includegraphics[scale=.7]{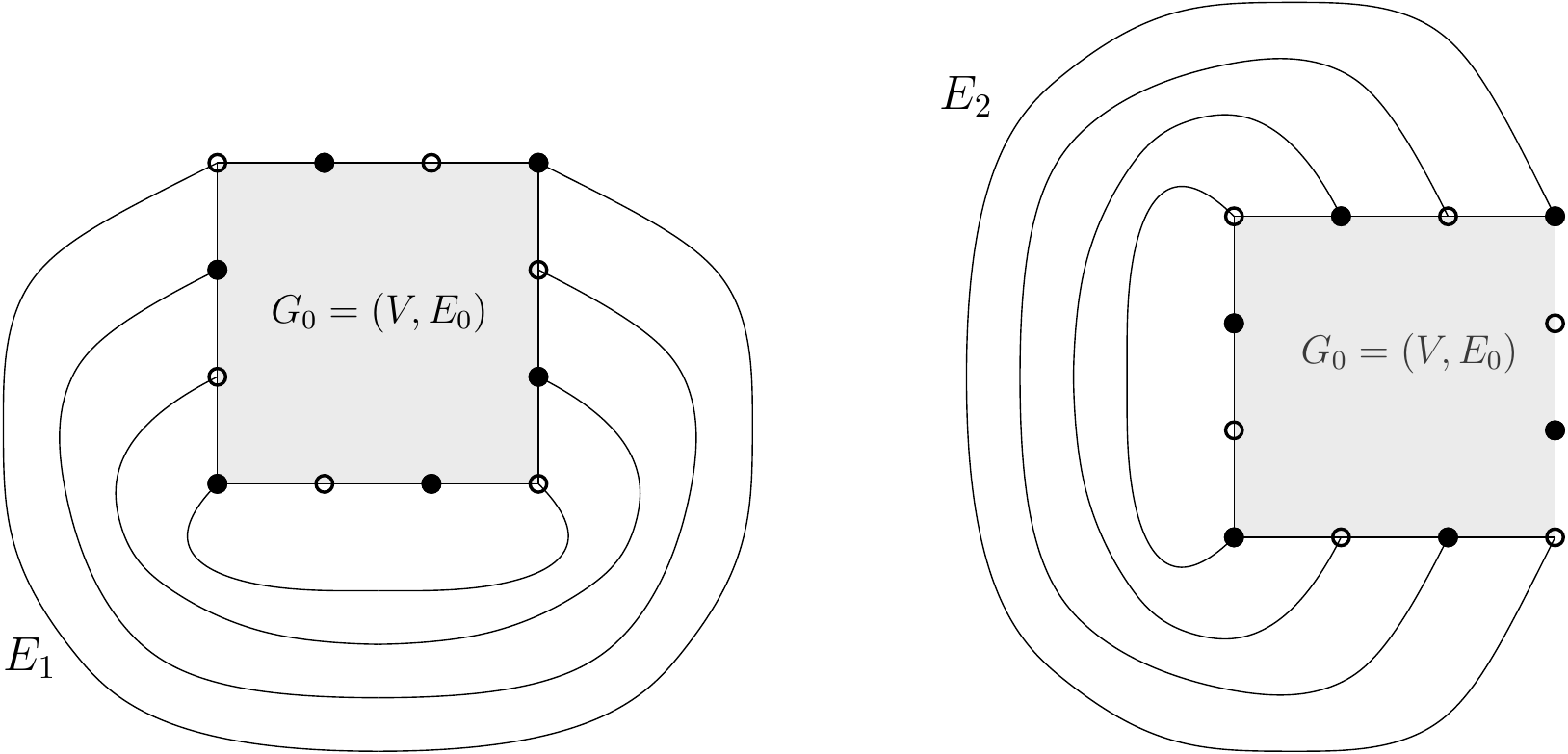}
\caption{The basis graph $G_0$ is schematically represented by the gray square (the vertices and edges inside the square are not shown). In the left (resp. right) drawing is pictured the planar graph $G_1$ (resp. $G_2$)}\label{fig:6}
\end{figure}

Note that there is a
unique orientation $D_j$ of the edges in $E_0\cup E_j$ that coincides with $D_0$ on $E_0$ and such that all the
internal faces of $G_j$ are clockwise odd. Then, we define the
relevant orientation $D_\th$ of type $\th=(\th_1,\th_2)\in \{-1,+1\}^2$ of $G$
as the unique orientation of the edges $E$ that coincides with $D_0$
on $G_0$ and with $\th_1 D_j$ on the edges in $E_j, j=1,2$.  Given one
of the four relevant orientations $D_\th$ of $G$, we define a
$|V|\times |V|$ antisymmetric matrix $A_{D_\th}$ by establishing that
for $v,v'\in V$, $A_{D_\th}(v,v')=0$ if $(v,v')\not\in E$, while
$A_{D_\theta}(v,v')=t_e$ if $v,v'$ are the endpoint of the edge $e$ oriented
from $v $ to $v'$, and $A_{D_\th}(v,v')=-t_e$ if $e$ is oriented from
$v'$ to $v$.  Then, \cite[Corollary 3.5]{Galluccio} says that
\begin{eqnarray}
  \label{eq:ZlibPf}
  Z_G=\sum_{M\in \Omega_G}w(M)= \sum_{\th \in \{-1,+1\}^2}\frac{c_\th}2\frac{{\rm Pf}(A_{D_\th})}{s(M_0)}
\end{eqnarray}
where 
 $\Omega_G$ is the set of the perfect matchings of $G$, $w(M)=\prod_{e\in M}t_e$, and
\begin{eqnarray}
  \label{eq:c}
c_{(-1,-1)}=-1 \text{ and } c_\th=1 \text{ otherwise.}
\end{eqnarray}
In \eqref{eq:ZlibPf}, ${\rm Pf}(A)$ denotes the Pfaffian of an anti-symmetric matrix $A$ and $s(M_0)$ denotes the sign of the term corresponding to the reference matching $M_0$ in the 
expansion of the Pfaffian ${\rm Pf}(A_{D_\th})$. Since by assumption $M_0$ contains only edges from $E_0$ whose orientation does not depend on $\theta$, $s(M_0)$ is indeed independent of $\theta$.

In our case, in contrast with the general case considered in
\cite{Galluccio}, the graph $G$ is bipartite. By labeling the vertices
so that the first $|V|/2$ are black and the last $|V|/2$ are white,
the matrices $A_{D_\th}(v,v')$ have then a block structure of the type
\begin{eqnarray}
  \label{eq:block}
A_{D_\th}= \begin{pmatrix}
0  & \rvline & +K_\th \\
\hline
  -K_\th & \rvline &
  0
\end{pmatrix}
\end{eqnarray}
We view the $|V|/2\times |V|/2$ ``Kasteleyn matrices'' $K_\th$ as having rows indexed by black vertices and columns by white vertices.

By using the relation \cite[Eq. (16)]{Haber} between Pfaffians and determinants, one can then rewrite the above formula as
\begin{eqnarray}
  \label{eq:Zlib}
  Z_G= \sum_{\th \in \{-1,+1\}^2} 
  \frac{\tilde c_\th}2\det(K_\th), \quad \tilde c_\th=c_\th\frac{(-1)^{(|V|/2-1)|V|/4}}{s(M_0)}.
\end{eqnarray}
\begin{remark}
  \label{rem:sceltasegno}
Note that changing the order in the labeling of the vertices changes the sign $s(M_0)$. We suppose henceforth that the choice is done so that the ratio in the definition of $\tilde c_\th$ equals $1$, so that $\tilde c_\th=c_\th$.  
\end{remark}

\subsection{Thermodynamic limit of the planar dimer model}
\label{sec:KG0}

In the previous section, Kasteleyn's theory for rather general
toroidal bipartite graphs was recalled, without assuming any type of
translation invariance.  In this subsection, instead, we specialize to
$G=G^0_L$ (the periodized version of $\mathbb Z^2$ introduced in
Section \ref{sec:wnpdm}) and, as was the case there, we assume that
the edge weights are invariant under translations by multiples of
${\bf e_1,\bf e_2}$.

With Kasteleyn's theory at hand, one can compute
the thermodynamic and large-scale properties of the dimer model on
$G^0_L$ as $L\to\infty$. We refer to \cite{KOS,Kenyon_notes,Gorin}
for details.  In the case where $G=G^0_L$, the basis graph $G_0$ is a
square grid with $Lm$ vertices per side and we choose its basic
orientation $D_0$ so that horizontal edges are oriented from left to
right, while vertical edges are oriented from bottom to top on every
second column and from top to bottom on the remaining columns. With
this choice, the orientations $D_1,D_2$ of $G_1,G_2$ are like in
Fig.  \ref{fig:7}.  Note that, if $e=(b,w)\in E_L^0$ is an edge of
$G_L^0$, then for $\th=(\th_1,\th_2)\in\{-1,+1\}^2$, $K_{\th}(b,w)$ equals
$K_{(+1,+1)}(b,w)$ multiplied by $(-1)^{(\th_1-1)/2}$ if $e$ belongs to
$E_1$ (see Fig. \ref{fig:6}) and by $(-1)^{(\th_2-1)/2}$ if $e$ belongs to
$E_2$.
\begin{figure}[H]
\centering
\includegraphics[scale=.6]{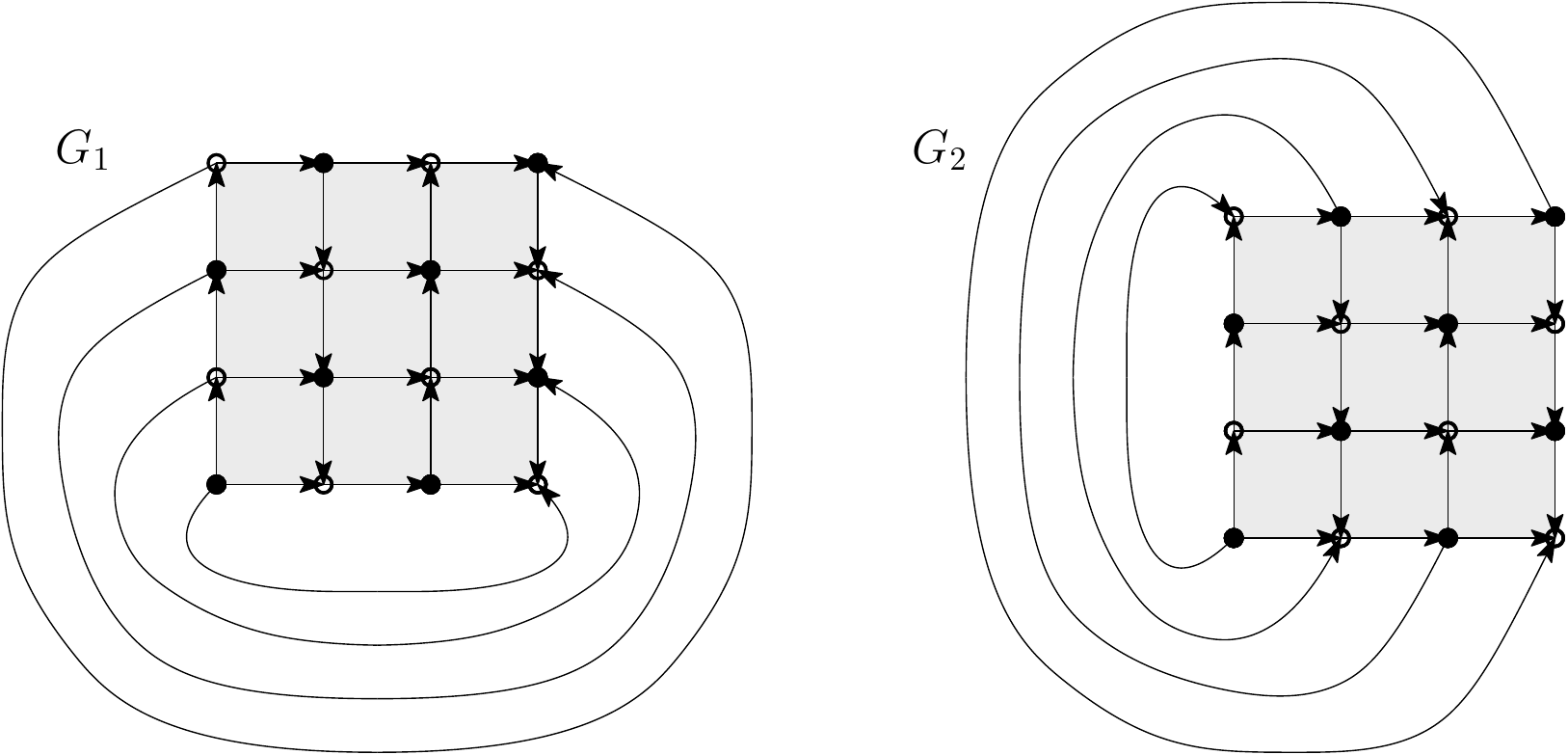}
\caption{The graphs $G_1,G_2$ corresponding to the basis graph of $G_L^0$ (for $L\,m=4$), together with their orientations $D_1,D_2$.}\label{fig:7}
\end{figure}
Observe also that the matrix $K_{(-1,-1)}$ is invariant under translations by multiples of ${\bf e_1,e_2}$. Define
\begin{eqnarray}
  \label{eq:Dr}
  \mathcal{P}(\th):=\left\{k=(k_1,k_2): k_j=\frac{2\pi}L\left(n_j+\frac{\th_j+1}4\right), -L/2<n_j\le L/2\right\}.
\end{eqnarray}
Let $P_\th$ be the orthogonal $(Lm)^2/2\times (Lm)^2/2$ matrix whose columns are indexed by $(k,\ell), k\in \mathcal{P}(\th), \ell\in \mathcal I=\{1,\dots,m^2/2\}$, whose rows are indexed by $(x,\ell), x\in \Lambda, \ell\in \mathcal I$, and such that  the column indexed $(k,\ell)$ is the vector
\begin{eqnarray}
  \label{eq:fka}
  f_{\ell,k}:((x,\ell')\in\Lambda\times \mathcal I)\mapsto f_{\ell,k}(x,\ell')=\frac1L e^{-i k x}{\bf 1}_{\ell'=\ell}.
\end{eqnarray}
Then, 
$P_\th^{-1}K_\th P_\th$ is block-diagonal with blocks of size $|\mathcal I|$ labelled by $k\in \mathcal{P}(\th)$.
The block corresponding to the value $k$ is a $|\mathcal I|\times |\mathcal I|$ matrix
$M(k)$ of elements $[M(k)]_{\ell,\ell'}$ with $\ell,\ell'\in \mathcal I$ and 
  \begin{eqnarray}
    \label{eq:Mni} [M(k)]_{\ell,\ell'}=\sum_{e:\ell\stackrel e\sim \ell'} K_{(-1,-1)}(b,w) e^{-i k
x_{e}}.
  \end{eqnarray} In this formula, the sum runs over all edges $e$ joining
the black vertex $b$ of type $\ell$ in the cell of
coordinates $x=(0,0)$ to some white vertex $w$ of type $\ell'$ ($w$ can be
either in the same fundamental cell or in another one); 
$x_{e}\in \mathbb Z^2$ is the coordinate of the cell to
which $w$ belongs.

The thermodynamic and large-scale properties of the measure
$\mathbb P_{L,0}$ are encoded in the matrix $M$: for instance the
infinite volume free energy exists  and it is given
by \cite{KOS}
\begin{align} F=:\lim_{L \to \infty}\frac1{L^2}\log
  Z_{L,0}=\frac{1}{(2\pi)^2} \int_{[-\pi,\pi]^2} \log | \mu(k) | dk
\end{align} where $\mu$ (the ``characteristic polynomial'') is
\begin{eqnarray}
  \label{eq:mu}
\mu(k):=\det M(k),  
\end{eqnarray}
which is a polynomial in $e^{i k_1},e^{i k_2}$. Kasteleyn's theory allows one to write
multi-point dimer correlations (in the $L\to\infty$ limit) in terms of  the so-called ``infinite-volume inverse Kasteleyn matrix'' $K^{-1}$: if $w$ (resp. $b$) is a white (resp. black) vertex of type $\ell $ in cell $x=(x_1,x_2)\in \mathbb Z^2$ (resp. of type $\ell'$ and in cell $0$), then one has
\begin{eqnarray}
  \label{eq:invKast}
  K^{-1}(w,b):=\frac1{(2\pi)^2}\int_{[-\pi,\pi]^2} 
  [(M(k))^{-1}]_{\ell,\ell'}e^{-i k x}  dk.
\end{eqnarray}
As can be guessed from \eqref{eq:invKast}, the long-distance
behavior of $K^{-1}$ is related to the zeros of the determinant of $M(k)$, that is, to the zeros of $ \mu$ on
$[-\pi,\pi]^2$.  It is a well known fact \cite{KOS} that, for any choice
of the edge weights, $\mu$ can have at most two zeros.  Our Assumption
\ref{ass:1} means that we restrict to a choice of edge weights such
that $\mu$ has exactly two  zeros, named $p_0^+,p_0^-$, {with $p_0^+\ne p^-_0\mod(2\pi,2\pi)$}.  We also define the complex numbers
\begin{equation}
  \label{eq:alphabeta}
  \alpha^0_\omega:=\partial_{k_1}\mu(p_0^\omega), \quad   \beta^0_\omega:=\partial_{k_2}\mu(p_0^\omega), \quad \omega=\pm.
\end{equation}
Note that,
since the Kasteleyn matrix elements $K_\th(b,w)$ are real\footnote{In \cite{GMT20,GMTHaldane} etc, a different choice of Kasteleyn matrix was done, with complex entries. As a consequence, in that case one had $p_0^++p_0^-=(\pi,\pi)$ instead.}, from  \eqref{eq:Mni} we have
the symmetry \begin{equation}\label{eq:symm0} [M(-k)]_{\ell,\ell'}=\overline{[M(k)]_{\ell\ell'}}
             \end{equation}
             and in particular
\begin{align}\label{fermippvel}
& p_0^++p_0^-=0  \\
& \alpha^0_-=-\overline{\alpha^0_+}, \quad \beta^0_-=-\overline{\beta^0_+}.
\label{fermippvel2}\end{align}
It is also known \cite{KOS} that $\alpha^0_\omega,\beta^0_\omega $ are not collinear as elements of  the complex plane:
\begin{eqnarray}
  \label{eq:alphabeta2}
  \alpha^0_\omega/\beta^0_\omega\not\in \mathbb R.
\end{eqnarray}
Note that from \eqref{fermippvel2} it follows that $\text{Im}(\beta_+^0/\alpha_+^0)=-\text{Im}(\beta_-^0/\alpha_-^0)$. From now on, with no loss of generality, we assume that 
\begin{equation} \label{signab} \text{Im}(\beta_+^0/\alpha_+^0)>0, \end{equation}
which amounts to choosing appropriately the labels $+,-$ associated with the two zeros of $\mu(k)$.

If we denote by $\adj(A)$ the adjugate of the matrix $A$, so that $A^{-1}=\adj(A)/\det A$, 
the long-distance behavior of the inverse Kasteleyn matrix is given \cite{KOS} as
\begin{eqnarray}
  \label{eq:KinvAsympt}
   K^{-1}(w,b)\stackrel{|x|\to\infty}=\frac1{2\pi}\sum_{\omega=\pm}[\adj (M(p^\omega))]_{\ell,\ell'}\frac{e^{-i p_0^\omega x}}{\phi^0_\omega(x)}+O(|x|^{-2})
\end{eqnarray}
where
\begin{eqnarray}
  \label{eq:phiomega}
  \phi^0_\omega(x)=\omega(\beta^0_\omega x_1-\alpha^0_\omega x_2).
\end{eqnarray}
Note that since the zeros $p_0^\omega$ of $\mu(k)$ are simple, the matrix $\adj M(p_0^\omega)$ has rank $1$. This means that we can write
\begin{eqnarray}\label{krondecM}
\adj M(p_0^\omega)=U^\omega \otimes V^\omega
\end{eqnarray}
for vectors $U^\omega,V^\omega \in \mathbb{C}^{|\mathcal{I}|}$, where
$\otimes$ is the Kronecker product.  Let $e=(b,w),e'=(b',w')$ be two
fixed edges of $G^0_L$: we assume that the black endpoint of $e$
(resp. of $e'$) has coordinates ${\bf x}=(x,\ell)$ (resp.
${\bf x'}=(x',\ell')$) and that the white endpoint of $e$ (resp. $e'$)
has coordinates $(x+v(e),m)$ with $m\in \mathcal I$ (resp. coordinates
$(x'+v(e'),m')$). Of course, $v(e)$ is either $(0,0)$ or $(0,\pm1)$ or
$(\pm1,0)$, and similarly for $v(e')$. Note that the coordinates of
the white endpoint of $e$ are uniquely determined by the coordinates of the
black endpoint and the orientation label\footnote{recall the conventions on labeling the type of edges, in Section \ref{sec:wnpdm}.}  $j \in \{1,\dots,4\}$ of $e$: in
this case we will write $v(e)=:v_{j,\ell}$, $K(b,w)=:K_{j,\ell}$ and
in \eqref{krondecM}, $U_m=:U_{j,\ell}$.  The
(infinite-volume) truncated dimer-dimer correlation under the measure
$\mathbb P_{L,0}$ is given as\footnote{the index $\th\in\{-1,+1\}^2$ in
  $K_\th(b,w)$ is dropped, since the dependence on $r$ is present only
  for edges at the boundary of the basis graph $G_0$ (see Figure
  \ref{fig:7}, so that for fixed $(b,w)$ and $L$ large, $K_\th(b,w)$ is
  independent of $r$}
\begin{eqnarray}
  \label{eq:dimdim0}
\mathbb E_{0}(\openone_{e};\openone_{e'})=\lim_{L\to\infty}  \mathbb E_{L,0}(\openone_{e};\openone_{e'})=-K(b,w)K(b',w')K^{-1}(w',b)K^{-1}(w,b').
\end{eqnarray}

As a consequence of the asymptotic expression \eqref{eq:KinvAsympt}, we have that as $|x'-x| \to \infty$,
\begin{align}\label{asymcorrfunc}
\mathbb{E}_0[\openone_e;\openone_{e'}]=A_{j,\ell,j',\ell'}(x,x')+B_{j,\ell,j',\ell'}(x,x')+R^0_{j,\ell,j',\ell'}(x,x')
\end{align}
with
\begin{equation}
\begin{split} A_{j,\ell,j',\ell'}(x,x')&= \sum_{\omega=\pm} \frac{ K^0_{\omega,j,\ell}  K^0_{\omega,j',\ell'}}{ (\phi^0_\omega(x-x'))^2} \\
B_{j,\ell,j',\ell'}(x,x')&= \sum_{\omega=\pm} \frac{H^0_{\omega,j,\ell}H^0_{-\omega,j',\ell'}}{|\phi^0_\omega(x-x')|^2}e^{2ip_0^\omega (x-x')}  \\
|R^0_{j,\ell,j'\ell'}(x,x')|&\leq C |x-x'|^{-3}.\end{split}\label{asym0}
\end{equation}
where
\begin{equation}\begin{split}
& K^0_{\omega,j,\ell}:=\frac1{2\pi}K_{j,\ell}e^{-ip_0^\omega v_{j,\ell}}
U^\omega_{j,\ell}V^\omega_\ell \\ 
& H^0_{\omega,j,\ell}:=\frac1{2\pi}K_{j,\ell}
e^{ip_0^\omega v_{j,\ell}} U^{-\omega}_{j,\ell}V^{\omega}_{\ell}.
\end{split}\end{equation}

\subsection{A fermionic representation for $Z_{L,\lambda}$}\label{sec:3.3}
In this subsection, we work again with generic edge weights, i.e., we do not assume that they have any spatial periodicity.

\subsubsection{Determinants and Grassmann integrals}

We refer for instance to \cite{GM01} for an introduction to Grassmann
variables and Grassmann integration; here we just recall a few basic
facts.  To each vertex $v$ of $G_L$ we associate a Grassmann
variable. Recall that vertices are distinguished by their color and by coordinates
${\bf x}=(x,\ell)\in{\bf \Lambda}= \Lambda\times \mathcal I$. We
denote the Grassmann variable of the black (resp. white) vertex of
coordinate ${\bf x}$ as $\psi^+_{\bf x}$ (resp.  $\psi^-_{\bf x}$).
We denote by $\int D\psi f(\psi)$ the Grassmann integral of a function
$f$ and since the variables $\psi^\pm_{\bf x}$ anti-commute among
themselves and there is a finite number of them, we need to define the
integral only for polynomials $f$. The Grassmann integration is a
linear operation that is fully defined by the following conventions:
\begin{equation}
  \label{eq:Dpsi}
\int D\psi\,
\prod_{{\bf x}\in{\bf \Lambda}}\psi^-_{\bf x}\psi^+_{\bf x}=1  ,
\end{equation}
the sign of the integral changes whenever the positions of two
variables are interchanged (in particular, the integral of a monomial
where a variable appears twice is zero) and the integral is zero if
any of the $2|{\bf \Lambda}|$ variables is missing. We also consider
Grassmann integrals of functions of the type $f(\psi)=\exp(Q(\psi))$,
with $Q$ a sum of monomials of even degree. By this, we simply mean
that one replaces the exponential by its finite Taylor series
containing only the terms where no Grassmann variable is repeated.

For the partition function $Z_{L,0}=Z_{G_L^0}$ of the dimer model on
$G_L^0$ we have formula \eqref{eq:Zlib} of previous subsection where the
Kasteleyn matrices $K_\th$ are fixed as in Section \ref{sec:KG0}, recall also Remark \ref{rem:sceltasegno}.  Using the
standard rewriting of determinants as Gaussian Grassmann integrals (i.e. Grassmann integrals where the integrand is the exponential of the corresponding quadratic form),
one immediately obtains
  \begin{equation}
    \label{eq:ZlibGrass}
  Z_{L,0}=\frac{1}{2}\sum_{\th  \in\{-1,+1\}^2}c_\th\int D\psi\, e^{-\psi^+K _\th  \psi^-}.
\end{equation}

\subsubsection{The partition function as a non-Gaussian, Grassmann integral}

The reason why the r.h.s. of \eqref{eq:Zlib} is the sum of four
determinants (and $Z_{L,0}$ is the sum of four Gaussian Grassmann
integrals) is that $G_L^0$ is embedded on the torus, which has genus
$1$: for a dimer model embedded on a surface of genus $g$, the
analogous formula would involve the sum of $4^g$ such determinants
\cite{Galluccio,Tesler}. This is clearly problematic for the graph
$G_L$ with non-planar edges, since in general it can be embedded only
on surfaces of genus $g$ of order $L^2$ (i.e. of the order of the
number of non-planar edges) and the resulting formula would be practically useless
for the analysis of the thermodynamic limit.  Our first crucial result
is that, even when the weights of the non-planar edges $N_L$ are
non-zero, the partition function can again be written as the sum of just four
Grassmann integrals, but these are non Gaussian (that is, the
integrand is the exponential is a polynomial of order higher than
$2$).
To emphasize that the following identity holds for generic edge weights, we will write $Z_{L,\underline t}$ for the partition function. 

\begin{proposition}\label{prop:1}
One has the identity
\begin{equation}\label{eq:ZGrass}
  Z_{L,\underline t}=\sum_{M\in \Omega_L}\prod_{e\in M}t_e=\frac{1}{2}\sum_{\th \in \{-1,+1\}^2} c_\th \int D\psi e^{-\psi^+ K_{\th} \psi^-+V_{\underline t}(\psi)}\end{equation}
where $c_\th$ are given in \eqref{eq:c},
$\Lambda=(-L/2,L/2]^2\cap\mathbb Z^2$ as above,
\begin{equation}
  \label{eq:Vt}
  V_{\underline t}(\psi)=\sum_{x \in \Lambda} V^{(x)}(\psi|_{B_x})
\end{equation}
and $V^{(x)}$ is a polynomial with coefficients depending on the weights of the edges incident to the cell $B_x$, $\psi|_{B_x}$ denotes the collection of the variables $\psi^\pm$
associated with the vertices of  cell  $B_x$ (as a consequence,
the order of the polynomial is at most $m^2$). When the edge weights
$\{t_e\}$ are invariant by translations by ${\bf e_1,e_2}$, then $V^{(x)}$
is  independent of $x$.
\end{proposition}
The form of the polynomial $V^{(x)}$ is  given in formula
\eqref{eq:Vx} below; the expression in the r.h.s. can be computed easily when
either the cell size $m$ is small, or each cell contains a small number of
non-planar edges. For an explicit example, see Appendix \ref{app:B}.

\begin{proof}
  We need some notation. 
  If $(b,w)$ is a pair of black/white vertices  joined by the edge $e$ of weight $t_e$, let us set
  \begin{equation}
    \label{eq:psie}
    \psi_\th (e):=\left\{
    \begin{array}{lll}
      -t_e \psi^+_b\psi^-_w &\text{ if }& e\in N_L\\
       -K_\th(e) \psi^+_b\psi^-_w &\text{ if } &e\in E^0_L
    \end{array}
                                                 \right.
  \end{equation}
  with $K_\th(e)=K_\th(b,w)$ the Kasteleyn matrix element corresponding to
  the pair $(b,w)$, which are the endpoints of $e$.  We fix a reference dimer configuration
  $M_0\in \Omega^0_L$, say the one where all horizontal edges
  of every second column are occupied, see Fig. \ref{fig:2}.
 
Then, we draw
the non-planar edges on the two-dimensional torus on which $G^0_L$ is
embedded, in such a way that they do not intersect any edge in $M_0$
(the non-planar edges will in general intersect each other and will
intersect some edges in $E^0_L$ that are not in $M_0$). 
Given $ J\subset N_L$, we let $P_{J}$ be the set of edges in $E^0_L$
that are intersected by edges in $J$.  The drawing of the non-planar
edges can be done in such a way that resulting picture is still
invariant by translations of ${\bf e_1,e_2}$, the non-planar edges do
not exit the corresponding cell and the graph obtained by removing the
edges in $N_L\cup P_{N_L}$ (i.e. all the non-planar edges and the
planar edges crossed by them) is $2-$connected.  See Figure
\ref{fig:2}.

\begin{figure}[H]
\centering
\includegraphics[scale=1]{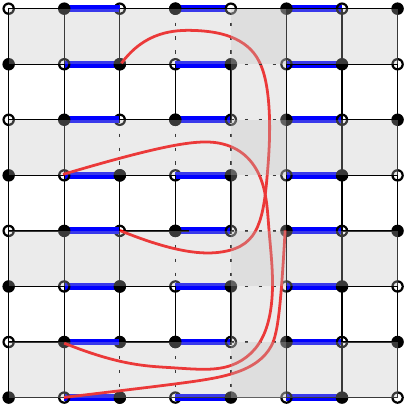}
\caption{A single cell $B_x$, with the reference configuration $M_0$
  (thick, blue edges). The non-planar edges (red) are drawn in a way that
  they do not intersect the edges of $M_0$ and do not exit the
  cell. Note that non-planar edges can cross each other. The dotted
  edges, crossed by the planar edges, belong to $P_{N_L}$. If the
  non-planar edges cross only horizontal
  edges in the same column (shaded) of the cell and vertical edges from every second row (shaded), the graph obtained by removing red edges and
  dotted edges is $2$-connected.}\label{fig:2}
\end{figure}

We start by rewriting 
\begin{equation}
  \label{eq:ZJS0}
  Z_{L,\underline t}=\sum_{J\subset N_L}\sum_{S\subset P_{J}}\sum_{M\in \Omega_{J,S}}w(M)
\end{equation}
where $\Omega_{J,S}$ is the set of dimer configurations $M$ such that a non-planar edge belongs to $M$ iff it belongs to $J$, and an edge in $P_J$ belongs to $M$ iff it belongs to $S$.
Given $M\in \Omega_{J,S}$, we write $M$ as the disjoint union $M=J\cup S\cup M'$ and, with obvious notations, $w(M)=w(M')w(S)w(J)$ so that \eqref{eq:ZJS} becomes
\begin{equation}
  \label{eq:ZJS}
  Z_{L,\underline t}=\sum_{J\subset N_L}w(J)\sum_{S\subset P_{J}}w(S)\sum_{M'\sim {J,S}}w(M')
\end{equation}
where $M'\sim S,J$ means that $M'\cup S\cup J$  is a dimer configuration in $\Omega_{J,S}$.
To proceed, we use the following
\begin{lemma}
  There exists $\epsilon_S^J=\pm 1$ such that  \begin{equation}
                                                 \label{eq:11}
    \sum_{M'\sim J,S}w(J)w(S)w(M')= \epsilon_S^J\sum_{\th \in\{-1,+1\}^2}\frac{c_\th}2\int D\psi\,
    e^{-\psi^+ K_\th \psi^-}\prod_{e\in J\cup S}\psi(e) .
  \end{equation}
  Here, $\psi(e), e\in J\cup S$ is the same as $\psi_\th(e)$: we have {removed the index $\theta$} because, 
since the endpoints $b,w$ of $e$ belong to the same cell, the right hand side of \eqref{eq:psie} is independent of $\theta$. If $J=S=\emptyset$, the product of $\psi(e)$ in the right hand side of \eqref{eq:11} should be interpreted as being equal to $1$.  
  Moreover, $\epsilon_{\emptyset}^{\emptyset}=1$ and, letting $J_x$ (resp. $S_x$) denote the collection of edges in $J$ (resp. $S$) belonging to the cell $B_x, x\in\Lambda$, one has
  \begin{equation}
    \label{eq:epsfact}
    \epsilon^J_S=\prod_{x\in \Lambda} \epsilon^{J_x}_{S_x}.
  \end{equation}
  \label{lem:1}
\end{lemma}
Let us assume for the moment the validity of Lemma \ref{lem:1} and conclude the proof of Proposition \ref{prop:1}. 
Going back to \eqref{eq:ZJS}, we deduce that
\begin{equation}
  \label{eq:14}
  Z_{L,\underline t}=\sum_\th\frac{c_\th}2\int D\psi\,
  e^{-\psi^+K_\th\psi^-}\prod_{x\in \Lambda}\left[\sum_{J_x}\sum_{S_x\subset P_{J_x}}\epsilon^{J_x}_{S_x}\prod_{e\in J_x\cup S_x}\psi(e)\right].
\end{equation}
The expression in brackets in \eqref{eq:14} can be written as
\begin{equation}
  1+F_x(\psi)=e^{V^{(x)}(\psi|_{B_x})}
\end{equation}
where $F_x(\psi)$ is a polynomial in the Grassmann fields of the box $B_x$, such that $F_x(0)=0$ and containing only monomials of even degree, and
\begin{equation}
  \label{eq:Vx}
  V^{(x)}(\psi|_{B_x})=\sum_{n\ge1}\frac{(-1)^{n-1}}n \left(F_x(\psi)\right)^n.
\end{equation}
\end{proof}

\begin{proof}[Proof of Lemma \ref{lem:1}]
  First of all, let us define a $2-$connected graph $G_{J,S}$, embedded on the
  torus, obtained from $G_L$ as follows:
  \begin{enumerate}
  \item the edges belonging to $N_L\cup P_J$ are removed. At this point, every cell $B_x$ contains a certain number (possibly zero) of faces that are not elementary squares, and the graph is still $2$-connected, recall the discussion in the caption of Figure \ref{fig:2}.
  \item the boundary of every such non-elementary face $\eta$ contains
    an even number of vertices that are endpoints of edges in $J\cup
    S$. We connect these vertices pairwise via new edges that do not
    cross each other, stay within $\eta$ and have endpoints of
    opposite color. See Figure \ref{fig:4} for a description of a
    possible procedure. We let $E_{J,S}$ denote the collection of the
    added edges.
  \end{enumerate}

\begin{figure}[H]
\centering
\includegraphics[scale=.99]{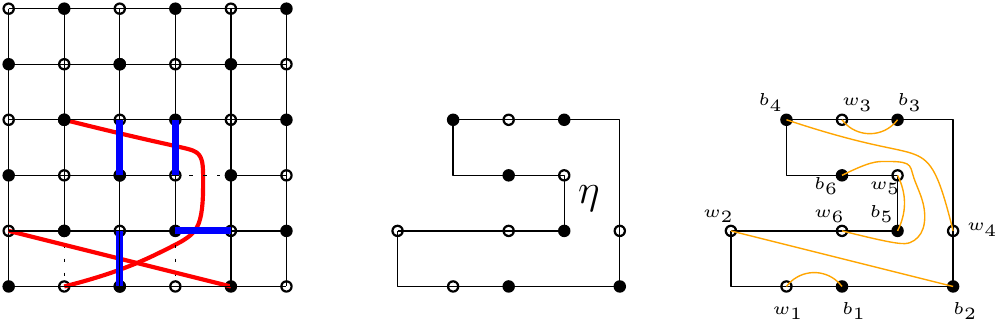}
\caption{Left drawing: a cell with a collection $J$ of non-diagonal
  edges (red) and of edges $S\subset P_J$ (thick blue edges). The
  dotted edges are those in $P_J\setminus S$. Center drawing: the
  non-elementary face $\eta$ obtained when the edges in $N_L\cup P_J$
  are removed. Only the endpoints of edges in $J\cup S$ are drawn.  Right drawing: a planar,
  bipartite pairing of the endpoints of $J\cup S$. The
  edges in $E_{J,S}$ are drawn in orange.  A possible
  algorithm for the choice of the pairing is as follows: choose arbitrarily a pair $(w_1,b_1)$ of
  white/black vertices that are adjacent along the boundary of $\eta$
  and pair them. At step $n>1$, choose arbitrarily a pair $(w_n,b_n)$
  that is adjacent once the vertices $w_i,b_i,i<n$ are
  removed. Note that some of the edges in $E_{J,S}$ may form double edges with the edges of $G^0_L$ on the boundary of $\eta$ (this is the case for $(b_1,w_1)$ and $(b_3,w_3)$ in the example in the figure).
  }\label{fig:4}
\end{figure}
  
  The first observation is that the l.h.s. of \eqref{eq:11} can be written as
  \begin{equation}\label{before}
   \Big(\prod_{e\in J\cup S} t_e \Big)\Big(\sum_{\substack{M\in \Omega_{G_{J,S}}:\\M\supset E_{J,S}}} w(M)\Big)\Big|_{t_{e}= 1, e\in E_{J,S}}
  \end{equation}
  where $\Omega_{G_{J,S}}$ is the set of perfect matchings of the graph
  $G_{J,S}$ and as usual $ w(M)$ is the product of the edge weights in
  $M$. The new edges $E_{J,S}$ are assigned a priori arbitrary weights $\{t_e\}_{e\in E_{J,S}}$, to be eventually replaced by $1$, {and the partition function on  $G_{J,S}$ is called $Z_{G_{J,S}}$.}

  Let $K^{J,S}_\th, \th\in\{-1,+1\}^2$ denote the  Kasteleyn matrices
  corresponding to the four relevant orientations $D_\th$ of $G_{J,S}$,
  for some choice of the basic orientation on $G_{J,S}$ (recall
  Definition \ref{def:bo}).  Since $G_{J,S}$ is embedded on the torus
  and is $2$-connected, Eq.\eqref{eq:Zlib} 
  guarantees that the sum in the second parentheses in \eqref{before} can be rewritten (before setting $t_e=1$ for all $e\in E_{J,S}$) as
\begin{eqnarray}
  \label{eq:17}
  \sum_{\substack{M\in \Omega_{G_{J,S}}\\M\supset E_{J,S}}} w(M){=
\left(\prod_{e\in E_{J,S}}t_e\partial_{t_e}\right)Z_{G_{J,S}}}=
  \frac12\sum_{\th}c_\th\left(\prod_{e\in E_{J,S}}t_e\partial_{t_e}\right)\det K^{J,S}_\th.
\end{eqnarray}
In fact, the suitable choice of ordering of vertices mentioned in Remark \ref{rem:sceltasegno} (and therefore the value of signs $c_\theta$) is independent of $J,S$, because the reference configuration $M_0$ is independent of $J,S$. 

Using the basic properties of Grassmann variables, the r.h.s. of \eqref{eq:17} equals
\begin{multline}
  \frac12\sum_{\th}c_\th\left(\prod_{e\in E_{J,S}}t_e\partial_{t_e}\right)\int D\psi\,
  e^{-\psi^+ K^{J,S}_\th \psi^-}\\
  =  \frac12\sum_{\th}c_\th\int D\psi\,
  e^{-\psi^+ K^{J,S}_\th \psi^-}\left(\prod_{e\in E_{J,S}}\psi^{J,S}_\th(e)\right)\label{eq:19}
\end{multline}
where, in analogy with \eqref{eq:psie}, $\psi^{J,S}_\th(e)=-K^{J,S}_\th(b,w)\psi^+_b\psi^-_w$.
We claim:
\begin{lemma}
The choice of the basic orientation of $G_{J,S}$ can be made so that the Kasteleyn  matrices $K_\th^{J,S}$ satisfy:
\begin{enumerate}
\item [(i)] if $e=(b,w)\in G_{J,S}\setminus E_{J,S}$, then $K^{J,S}_\th(b,w)= K_\th(b,w)$, with $K_\th$ the Kasteleyn
  matrices of the graph $G_L^{0}$, fixed by the choices explained in
  Section \ref{sec:KG0}.
\item [(ii)] if instead $e=(b,w)\in E_{J,S}$ and is contained in cell $B_x$, then $K^{J,S}_\th(b,w)=t_e
  \sigma_e^{J_x,S_x}$ with $  \sigma_e^{J_x,S_x}=\pm1$ a sign that depends only on $J_x,S_x$.
\end{enumerate}
\label{lem:KJS}
\end{lemma}

Assuming Lemma \ref{lem:KJS}, and letting $E_{J_x,S_x}$ denote the subset of edges in $E_{J,S}$ that belong to cell $B_x$, we rewrite \eqref{eq:19} as
\begin{equation} \label{eq:3.39}\frac12\sum_{\th}c_\th\int D\psi\,
  e^{-\psi^+ K_\th \psi^-}\prod_x\prod_{e=(b,w)\in E_{J_x,S_x}}(-t_e\sigma_e^{J_x,S_x}\psi^+_b\psi^-_w),\end{equation}
where we could replace $K^{J,S}_\theta$ by $K_\theta$ at exponent, because 
\begin{equation}
\begin{split} & e^{-\psi^+ K^{J,S}_\th \psi^-}\left(\prod_{e\in E_{J,S}}\psi^{J,S}_\th(e)\right)
=\Big(\prod_{e=(b,w)\in G_{J,S}\setminus E_{J,S}}e^{-\psi^+_b K^{J,S}_\th(b,w) \psi^-_w}\Big)\left(\prod_{e\in E_{J,S}}\psi^{J,S}_\th(e)\right)\\
& =\Big(\prod_{e=(b,w)\in G_{J,S}\setminus E_{J,S}}e^{-\psi^+_b K_\th(b,w) \psi^-_w}\Big)\left(\prod_{e\in E_{J,S}}\psi^{J,S}_\th(e)\right)
=e^{-\psi^+ K_\th \psi^-}\left(\prod_{e\in E_{J,S}}\psi^{J,S}_\th(e)\right),\end{split}
\end{equation}
thanks to the Grassmann anti-commutation properties and the fact that $K^{J,S}_\theta(b,w)=K_\theta(b,w)$ for any $(b,w)\in G_{J,S}\setminus E_{J,S}$.
Eq.\eqref{eq:3.39} can be further rewritten as 
  \begin{equation}
\frac{\prod_{e\in E_{J,S}}t_e}{\prod_{e\in J\cup S}t_e}\sum_{\th}\frac{c_\th}2 \int D\psi\,
  e^{-\psi^+ K_\th \psi^-}\prod_x \Big(\epsilon^{J_x}_{S_x}\prod_{e\in J_x\cup S_x}\psi(e)\Big),
  \end{equation}
  where $\epsilon^{J_x}_{S_x}$ is a sign, equal to
\begin{equation}
  \pi(J_x,S_x)\Big(\prod_{e\in E_{J_x,S_x}}\sigma_e^{J_x,S_x}\Big)\Big(\prod_{e\in S_x}{\rm sign}(K_\th(e))\Big),
  \label{veryverysimple}
\end{equation}
and $\pi(J_x,S_x)$ is the sign of the permutation needed to recast $\prod_{(b,w)\in E_{J_x,S_x}}\psi^+_b\psi^-_w$ into the form $\prod_{(b,w)\in J_x\cup S_x}\psi^+_b\psi^-_w$;
note also that, for $e\in S_x$, $K_\th(e)$ is independent of $\theta$. Putting things together, the statement of  Lemma \ref{lem:1} follows. 
\end{proof}

\begin{proof}[Proof of Lemma \ref{lem:KJS}]
  Recall that $G_{J,S}$ is a $2$-connected graph, with the same vertex set as $G^0_L$, and edge set obtained, starting from $E_L$, by removing 
  the edges in $N_L\cup P_J$ and by adding those in $E_{J,S}$.  We introduce a sequence of $2$-connected graphs
  $G^{(n)},n=0,\dots, z= |E_{J,S}|$ embedded on the torus,  all
  with the same vertex set.  Label the edges in $E_{J,S}$
  as $e_1,\dots,e_z$ (in an arbitrary order). Then, $G^{(0)}$ is the
  graph $G_L^0$ with the edges in $N_L\cup P_J$ removed and
  $G^{(n)},1\le n\le z$ is obtained from $G^{(0)}$ by adding edges
  $e_1,\dots,e_n$. Note that $G^{(z)}=G_{J,S}$. We will recursively
  define the basic orientation $D^{(n)}$  of $G^{(n)}$, in such a
  way that for $n=z$ the properties stated in the Lemma hold for the Kasteleyn matrices $K_\th^{(z)}=K^{J,S}_\th$.
   The construction of the basic orientation is
  such that for $n>m$, $D^{(n)}$ restricted to the edges of
  $G^{(m)}$ is just $D^{(m)}$. That is, at each step $n>1$ we
  just need to define the orientation of  $e_n$.

  For $n=0$, $G^{(0)}$ is a sub-graph of $G_L^0$ and we simply define
  $D^{(0)}$ to be the restriction of $D$ (the basic orientation of
  $G^0_L$) to the edges of the basis graph of $G^{(0)}$.  Since the
  orientation of these edges will not be modified in the iterative
  procedure, point (i) of the Lemma is automatically satisfied. We
  need to show that $D^{(0)}$ is indeed a basic orientation for
  $G^{(0)}$, in the sense of Definition \ref{def:bo}. In fact, an
  inner face $\eta$ of the basis graph of $G^{(0)}$ is either an
  elementary square face (which belongs also to the basis graph of
  $G_L^0$), or it is a non-elementary face as in the middle drawing of
  Fig. \ref{fig:4}. In the former case, the fact that the boundary of
  $\eta$ is clockwise odd is trivial, since its orientation is the same
  as in the basic orientation of $G_L^0$. In the latter case, the
  boundary of $\eta$ is a cycle $\Gamma$ of $\mathbb Z^2$ that
  contains no vertices in its interior. The fact that $\Gamma$ is clockwise odd for $D$ then is well-known \cite[Sect.V.D]{Kast67}.

 Assume now that the basic orientation $D^{(n)}$ of $G^{(n)}$ has been defined for $n\ge0$ and that the choice of orientation of each
  $e=(b,w)\in E_{J,S}$ that is an edge of $G^{(n)}$ contained in the cell
  $B_x$, has been done in a way that depends  on $J,S$ only
  through $J_x,S_x$. If $n=z$,  recalling how Kasteleyn matrices $K_\th$ are defined in terms of the orientations,
  claim (ii) of the Lemma is proven. Otherwise, we proceed to step
  $n+1$, that is we define the orientation of $e_{n+1}$ as explained in Figure \ref{fig:5}. This choice is unique and, again, depends on $J,S$ only
  through $J_x,S_x$. The proof of the Lemma is then concluded.
\end{proof}

  \begin{figure}[H]
\centering
\includegraphics[scale=1]{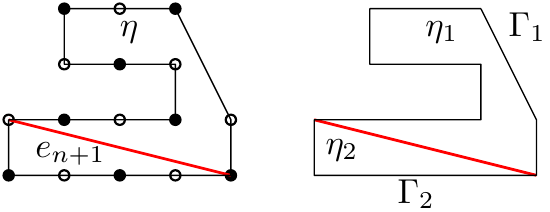}
\caption{An inner face $\eta$ of $G^{(n)}$ and the edge $e_{n+1}$. After adding $e_{n+1}$, $\eta$ split into  two inner faces $\eta_1,\eta_2$ of $G^{(n+1)}$. By assumption, the boundary $\Gamma$ of $\eta$ is
  clockwise-odd for the orientation $D^{(n)}$. Therefore, exactly one of the two paths $\Gamma_1,\Gamma_2$ contains an odd number of anti-clockwise oriented edges and there is a unique  orientation of $e_{n+1}$ such that the boundaries of both $\eta_1,\eta_2$ are clockwise odd. Since, by induction, the orientation of $\Gamma$ depends on $J,S$ only through $J_x,S_x$, with $x$ the label of the cell the face belongs to, the same is true also for the orientation of $e_{n+1}$.}
  \label{fig:5}
\end{figure}

\subsection{Generating function and Ward Identities}\label{sec:3.4}
In this subsection we consider again dimer weights that are periodic under translations by integer multiples of ${\bf e_1,\bf e_2}$.

In view of Proposition \ref{prop:1}, the generating function $W_L(A)$ of dimer correlations, defined, for $A: E_L \rightarrow \mathbb{R}$, by 
\begin{equation} e^{W_L(A)}:=\sum_{M\in\Omega_L} w(M) \prod_{e\in E_L} e^{A_e\mathds  1_e(M)}, \end{equation}
can be equivalently rewritten as $e^{W_L(A)}=\frac12\sum_{\theta\in\{1,-1\}^2}c_\theta e^{\mathcal W^{(\theta)}_L(A)}$, where
\begin{equation}\label{eq:ZGrassgen}
e^{\mathcal W^{(\theta)}_L(A)}=\int D\psi e^{S_\th(\psi)+V(\psi,A)}, 
\end{equation}
where $S_\th(\psi)=-\psi^+ K_\th\psi^-$ and $V(\psi,A):=
-\psi^+K_\th^A\psi_--S_\th(\psi)+V_{\underline t(A)}(\psi)$. Here, $K_\th^A$
(resp. $V_{\underline t(A)}(\psi)$) is the Kasteleyn matrix as in Section
\ref{sec:KG0} (resp. the potential as in
\eqref{eq:Vt}) with edge weights $\underline t(A)=\{t_e
e^{A_e}\}_{e\in E_L}$. 

As in \cite[Sect.3.2]{GMT20}, it is convenient to introduce a generalization of the generating function, in the presence of an external Grassmann field coupled with $\psi$. Namely, 
letting $\phi=\{\phi^\pm_{\bf x}\}_{{\bf x} \in {\bf \Lambda}}$ a new set of Grassmann variables, we define 
\begin{equation}\label{gener}\begin{split}
& e^{W_L(A,\phi)}:=\frac12\sum_{\theta\in\{1,-1\}^2}c_\theta e^{\mathcal W^{(\theta)}_L(A,\phi)},\\
\text{with}\qquad &e^{\mathcal W^{(\th)}_L(A,\phi)}:=\int D\psi\, e^{S_\th(\psi)+V(\psi,A)+(\psi,\phi)}\end{split}
\end{equation}
and $(\psi,\phi):=\sum_{{\bf x} \in{\bf \Lambda}}(\psi^+_{{\bf x}}\phi^-_{\bf x}+\phi^+_{\bf x}\psi^-_{\bf x})$. The generating function is invariant under a local gauge symmetry, 
which is associated with the local conservation law of the number of incident dimers at each vertex of $\Lambda$: 
\begin{proposition}[Chiral gauge symmetry]
Given two functions $\alpha^+ : {\bf \Lambda} \rightarrow \mathbb{R}$ and $\alpha^- :  {\bf \Lambda} \rightarrow \mathbb{R}$, we have
\begin{equation}\label{gaugeinvarianceeqn}
{W}_{L}(A,\phi)=-i\sum_{{\bf x} \in {\bf \Lambda}} (\alpha_{{\bf x}}^++\alpha_{\bf x}^-)+{W}_{L}(A+i\alpha,\phi e^{i\alpha})
\end{equation}
where, if $e=(b,w) \in E_L$ with ${\bf x}$ and ${\bf y}$ the coordinates of $b$ and $w$, respectively, $(A+i\alpha)_e:=A_e+i(\alpha^+_{\bf x}+\alpha^-_{\bf y})$,  
while $(\phi e^{i\alpha})^\pm_{\bf x}:=\phi^\pm_{\bf x} e^{i\alpha^\mp_{\bf x}}$. 
\end{proposition}
The proof simply consists in performing a change of variables in the Grassmann integral, see \cite[Proof of Prop.1]{GMTHaldane}. 

The gauge symmetry \eqref{gaugeinvarianceeqn}, in turn, implies exact identities among correlation functions, known as Ward Identities. Given edges $e_1,\ldots,e_k$ and a
collection of coordinates ${\bf x_1},\ldots,{\bf x_n},{\bf y_1},\ldots,{\bf y_n}$, define\footnote{We refer e.g. to \cite[Remark 5]{GMTHaldane} for the meaning of the
  derivative with respect to  Grassmann variables}
  the truncated multi-point correlation associated with the
generating function $\mathcal{W}_L(A,\phi)$:
\begin{equation}\label{multipointcorr}\begin{split}
& g_L(e_1,\ldots,e_k;{\bf x_1},\ldots,{\bf x_n};{\bf y_1},\ldots,{\bf y_n})\\
& \quad :=\partial_{A_{e_1}} \cdots \partial_{A_{e_k}}\partial_{\phi^-_{\bf y_1}} \cdots \partial_{\phi^-_{\bf y_n}} \partial_{\phi^+_{\bf x_1}} \cdots \partial_{\phi^+_{\bf x_n}}{W}_L(A,\phi) \big|_{A \equiv 0, \phi \equiv 0}. \end{split}
\end{equation}
Three cases will play a central role in the following: the {\it  interacting propagator} $G^{(2)}$, the
{\it interacting vertex function} $G^{(2,1)}$ and the {\it interacting dimer-dimer correlation} $G^{(0,2)}$, which deserve a distinguished notation: 
letting ${\bf x}=(x,\ell), {\bf y}=(y,\ell'), {\bf z}=(z,\ell'')$, 
and denoting by $e$ (resp. $e'$) the edge with black vertex ${\bf x}=(x,\ell)$ (resp. ${\bf y}=(y,\ell')$) and label $j\in\mathcal J_\ell$ (resp. $j'\in\mathcal J_{\ell'}$), we define
\begin{equation}\begin{split}
& G^{(2)}_{\ell,\ell';L}(x, y):=g_L(\emptyset;{\bf x};{\bf y}) \\
& G^{(2,1)}_{j,\ell,\ell',\ell'';L}({ x},{y},{ z}):=	g_L(e;{\bf y};{\bf z})\\
& G^{(0,2)}_{j,j',\ell,\ell';L}({x,y}):=g_L(e,e';\emptyset;\emptyset).\end{split}\label{notation}
\end{equation}
As a byproduct of the analysis of Section \ref{lasezionepesante}, the $L \to \infty$ of all multi-point correlations $g_L(e_1,\ldots,e_k,{\bf x_1},\ldots,{\bf x_n},{\bf y_1},\ldots,{\bf y_n})$ exist; we denote the limit simply by dropping the index $L$. Let us define the Fourier transforms of the interacting propagator and interacting vertex function via the following conventions:
for $\ell,\ell',\ell''\in \mathcal I$ and $j\in \mathcal J_{\ell}$, we let 
\begin{equation}
\label{eq:convFourier}\begin{split}
& \hat G^{(2)}_{\ell, \ell '}(p):=\sum_{x\in\mathbb Z^2} e^{ip  x} G^{(2)}_{\ell,\ell'}(x, 0) \\
& \hat G^{(2,1)}_{j,\ell,\ell ', \ell ''}(k,p):=\sum_{x,y\in\mathbb Z^2} e^{-ip  x-ik \cdot y}G^{(2,1)}_{j,\ell,\ell',\ell''}(x,0,y).
\end{split} \end{equation}
\begin{proposition}[Ward identity]\label{prop:WI}
Given $\ell ', \ell '' \in \mathcal{I}$, 
we have
\begin{equation}\label{FspWid}
	\sum_{\substack{e \in \mathcal{E}}}\hat{G}^{(2,1)}_{j(e),\ell(e),\ell',\ell''}(k,p)(e^{-ip \cdot v(e)}-1)=\hat{G}			^{(2)}_{\ell',\ell''}(k+p)-\hat{G}^{(2)}_{\ell',\ell''}(k)
\end{equation}
where $\mathcal{E}$ is the set of edges $e=(b(e),w(e))$ having an endpoint in the cell $B_{(0,0)}$ and the other in $B_{(0,-1)} \cup B_{(-1,0)}$. Also, $\ell(e)\in\mathcal I$ is the type of $b(e)$, $j(e)\in\mathcal J_{\ell(e)}$ is the label associated with the edge $e$, while $v(e)\in \{(0,\pm1),(\pm1,0)\}$ is the difference of cell labels of $w(e)$ and $b(e)$, see discussion after \eqref{krondecM}.
\end{proposition}

\begin{proof}
We start by differentiating both sides of the gauge invariance equation \eqref{gaugeinvarianceeqn}: fix $ {\bf x}=(x,\ell) \in {\bf \Lambda} $, differentiate first with respect to $\alpha^+_{\bf x}$ and set $\alpha \equiv 0$:
\begin{equation}\label{eq0}
	1=\sum_{\substack{e =(b,w) \in E_L \\ {\bf x}(b)={\bf x}}} \partial_{A_e} {W}_L(A,\phi)+\phi^-_{\bf x}\partial_{\phi^-_{\bf x}}{W}_L(A,\phi)
\end{equation}
where ${\bf x}(b)=(x(b),\ell(b))$ is the coordinate of the
black endpoint $b$ of the edge $e$. The above sum thus contains as
many terms as the number of edges incident to the black site of coordinate $
{\bf x}$, i.e. as the number of elements in $\mathcal J_{\ell(b)}$. Then, differentiate with respect to $\phi^-_{\bf z}$ and
$\phi^+_{\bf y}$ and set $A \equiv \phi \equiv 0$:
\begin{equation}\label{eq1}
\sum_{\substack{e=(b,w) \in E_L \\ {\bf x}(b)={\bf x}}} g_L(e;{\bf y};{\bf z})+\delta_{{\bf x, z}}g_L(\emptyset;{\bf y};{\bf z})=0.
\end{equation} 
Repeating the same procedure but differentiating first with respect to $\alpha^-_{\bf x}$ rather than $\alpha^+_{\bf x}$, 
we obtain
\begin{align}
	\sum_{\substack{e=(b,w) \in E_L \\ {\bf x}(w)={\bf x}}} g_{L}(e;{\bf  y};{\bf z})+\delta_{{\bf x, y}} g_L(\emptyset;{\bf y};{\bf z})=0  \label{eq2}
\end{align}
where ${\bf x}(w)$ is the coordinate of the white vertex of
$e$.  Now we sum both \eqref{eq2} and \eqref{eq1} over
$\ell \in \mathcal{I}$ (the type of the vertex ${\bf x}$) with the cell index $x$
fixed; then we take  the difference of the two expressions
thus obtained and we send $L \to \infty$. When taking the difference,
the contribution from edges whose endpoints both belong to cell $B_x$
cancel and we are left with 
\begin{align}
  \label{eq:330}
\sum_{e=(x',j,\ell)\in E_{\partial B_x}} (-1)^{\delta_{x,x'}} G^{(2,1)}_{j,\ell,\ell',\ell'}(x',y,z)=(\delta_{x,z}-\delta_{x,y})G^{(2)}_{\ell',\ell''}(y,z),
\end{align}
where we used the notation in \eqref{notation}, and we denoted by $E_{\partial B_x}$ the set of edges of $E^0$ having exactly one endpoint in the cell $B_x$.
Note that in the first sum, in writing $e=(x',j,\ell)$, we used 
the usual labeling of the edge $e$ in terms of the coordinates $(x',\ell)$ of its black site and of the label $j\in \mathcal J_\ell$. Note also that, 
if $e=(x',j,\ell)\in E_{\partial B_x}$, then $x'$ is either $x$ or $x\pm(0,1),x\pm(1,0)$. See Figure \ref{fig:9}.
 \begin{figure}[H]
\centering
\includegraphics[scale=1]{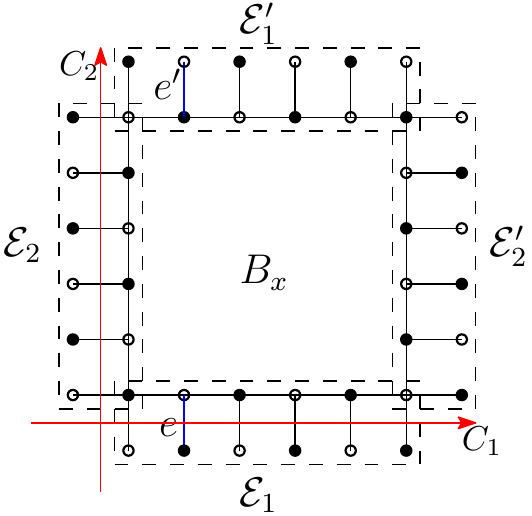}
\caption{The cell $B_x$ (only vertices on its boundary are drawn) together with the edges in $E_{\partial B_x}=\mathcal E_{1,x}\cup\mathcal E_{2,x}\cup \mathcal E'_{1,x}\cup \mathcal E'_{2,x}$. To each edge $e$ in $\mathcal E_{1,x}$ (resp. in $\mathcal E_{2,x}$) there corresponds a unique edge $e'$ in $\mathcal E_{1,x}'$ (resp. $\mathcal E_{2,x}'$) whose endpoints are of the same type.} 
  \label{fig:9}
\end{figure}
 Using the last remark in the caption of Fig.\ref{fig:9}, we can rewrite the sum in the left hand side of \eqref{eq:330} as a sum over edges in $\mathcal E_{1,x}\cup \mathcal E_{2,x}$, each term containing the difference of two vertex functions $G^{(2,1)}_{j,\ell,\ell',\ell''}$.  
Passing to Fourier space via \eqref{eq:convFourier}, we obtain \eqref{FspWid}, as desired. 
\end{proof}

\begin{remark}\label{relvsigma} For later reference, note that, if $e$ crosses the path $C_1$ (resp. $C_2$) of Figure \ref{fig:9}, i.e., if $e\in \mathcal E_{1,x}$ (resp. $e\in\mathcal E_{2,x}$), then, for any $p=(p_1,p_2)\in\mathbb R^2$
and $v(e)$ defined as in the statement of Proposition \ref{prop:WI}, 
\begin{equation}
    \label{eq:identita}
    p \cdot v(e)=
    \begin{cases}
      -p_2 \sigma_e& \text{ if $e$ crosses $C_1$}\\
      +p_1 \sigma_e& \text{ if $e$ crosses $C_2$},
    \end{cases}
  \end{equation}
  with $\sigma_e=\pm1$ the same sign appearing in the definition \eqref{heightfunction} of height function.
\end{remark}

\section{Proof of Theorem \ref{mainthrm}} \label{sec:4}

One important conclusion  of the previous section is Proposition \ref{prop:WI}, which states the validity of exact identities among the (thermodynamic limit of) correlation functions of the dimer model. In this section we combine these exact identities with a result on the large-distance asymptotics of the correlation functions, which includes the statement of Theorem \ref{thm:1}, and use them to prove 
Theorem \ref{mainthrm}. The required fine asymptotics of the correlation functions is summarized in the following proposition, 
whose proof is discussed in Section \ref{lasezionepesante}. 

\begin{proposition}\label{asimptrelprpgtr} There exists $\lambda_0>0$ such that, for $|\lambda|\le \lambda_0$, the interacting dimer-dimer correlation for $x\neq y$ 
can be represented in the following form:
\begin{eqnarray}\label{hh110}
	&&G^{(0,2)}_{j,j',\ell,\ell'}(x,y) = \frac{1}{4\pi^2 Z^2(1-\tau^2)}\sum_{\omega=\pm} \frac{K^{(1)}_{\omega,j,\ell} K^{(1)}_{\omega,j',\ell'}}{(\phi_\omega(x-y))^2}\\
	&&\qquad + \frac{B}{4\pi^2} \sum_{\omega=\pm} \frac{K^{(2)}_{\omega,j,\ell}  K^{(2)}_{-\omega,j',\ell'}}{| \phi_\omega(x-y)|^{2(1-\tau)/(1+\tau)}} e^{2i\, p^\omega\cdot(x-y)} 
	+R_{j,j',\ell,\ell'}(x,y)\nonumber\;,
\end{eqnarray}
where: $\lambda\mapsto Z$, $\lambda\mapsto\tau$ and $\lambda \mapsto B$ are real-valued analytic functions satisfying $Z=1+O(\lambda)$, $\tau=O(\lambda)$ and $B=1+O(\lambda)$;
$\phi_\omega(x):=\omega (\beta_\omega x_1-\alpha_\omega x_2)$ where $\lambda\mapsto \alpha_\omega,\lambda\mapsto \beta_\omega$ are complex-valued 
analytic functions satisfying $\overline{\alpha_+}=- \alpha_{-},  \overline{\beta_+}=- \beta_{-}$; $\lambda\mapsto K^{(i)}_{\omega,j,\ell}$ with $i\in\{1,2\}$ 
are complex-valued analytic functions of $\lambda$ satisfying 
$ K^{(i)}_{+,j,\ell}=\overline{K^{(i)}_{-,j,\ell}}$; $\lambda\mapsto p^\omega$ are analytic functions with values in 
$[-\pi,\pi]^2$ for $\lambda$ real, satisfying $p^+=-p^-$ and $2p^+\neq 0$ mod $(2\pi,2\pi)$; the correction term $R_{j,j',\ell,\ell'}(x,y)$ is translational invariant and satisfies $|R_{j,j',\ell,\ell'}(x,0)|\le C |x|^{-1+C|\lambda|}$ for some $C>0$.

Moreover, there exists an additional set of complex-valued analytic function  $\lambda \mapsto I_{\omega,\ell,\ell'},\omega=\pm1,\ell,\ell'\in\mathcal I,$ such that the Fourier transforms 
of the interacting propagator and of the interacting vertex function satisfy:  
\begin{equation} \hat G^{(2)}_{\ell,\ell'}(k+ p^\omega) \stackrel{ k \to 0}=  I_{\omega,\ell,\ell'} \hat G^{(2)}_{R,\omega}(k)[1+O(|k|^{1/2})],\label{h10ab}\end{equation}
and, if $0<\mathfrak c\le |p|,|k|,|k+p|\le 2\mathfrak c$,
\begin{equation}\label{h10a}
	\hat G^{(2,1)}_{j,\ell,\ell',\ell''}(k+ p^\omega, p)
	\stackrel{\mathfrak c\to0}= -\sum_{\omega'=\pm} K^{(1)}_{\omega',j,\ell}  I_{\omega,\ell',\ell''} \hat G^{(2,1)}_{R,\omega',\omega}(k,p)[1+O(\mathfrak c^{1/2})]\;,
\end{equation}
where $K^{(1)}_{\omega,j,\ell}$ is the same as in \eqref{hh110} and $\hat G^{(2)}_{R,\omega}(k)$,$\hat G^{(2,1)}_{R,\omega,\omega'}(k,p)$ are two functions satisfying, for $D_\omega(p)=\alpha_\omega p_1+\beta_\omega p_2$, 
\begin{equation}\label{h11}
\sum_{\omega'=\pm} D_{\omega'}(p)\hat G^{(2,1)}_{R,\omega',\o}(k,p)=
\frac1{Z(1-\tau)} \Big[\hat G^{(2)}_{R,\o}(k) - \hat G^{(2)}_{R,\o}(k+p)\Big]\Big(1+O(\lambda |p|)\Big)\;,\end{equation}
with $Z, \tau$ the same as  in \eqref{hh110}, and 
\begin{align}\label{symmrefprop}
\hat G^{(2,1)}_{R,-\omega,\omega}(k,p)=\tau \frac{D_{\omega}(p)}{D_{-\omega}(p)}\hat G^{(2,1)}_{R,\omega,\omega}(k,p)\Big(1+O(|p|)\Big).
\end{align}
Finally, $\hat G^{(2)}_{R,\omega}(k) \sim c_1|k|^{-1+O(\lambda^2)}$ as
$k\to 0$, and, if $0<\mathfrak c\le |p|,|k|,|k+p|\le 2\mathfrak c$,
$\hat G^{(2,1)}_{R,\omega,\omega'}(k,p) \sim c_2
\mathfrak{c}^{-2+O(\lambda^2)}$ as $\mathfrak{c} \to 0$, for two
suitable non-zero constants $c_1,c_2$.
\end{proposition}

\medskip

A few comments are in order. First of all, the statement of Theorem
\ref{thm:1}, \eqref{eq:asy}, follows from \eqref{hh110}, which is just
a way to rewrite it: it is enough to identify $K_{\omega,j,\ell}$ with
$(2\pi Z\sqrt{1-\tau^2})^{-1}K^{(1)}_{\omega,j,\ell}$,
$H_{\omega,j,\ell}$ with $(\sqrt{B}/2\pi)K^{(2)}_{\omega,j,\ell}$, and
$\nu$ with $(1-\tau)/(1+\tau)$.

Moreover, we emphasize that Proposition \ref{asimptrelprpgtr} is the analogue of
\cite[Prop.2]{GMT20} and its proof, discussed in the next section, is
a generalization of the corresponding one. The main ideas behind the
proof remain the same: in order to evaluate the correlation functions
of the non-planar dimer model we start from the Grassmann
representation of the generating function,
\eqref{gener}, and we compute it via an iterative
integration procedure, in which we first integrate out the degrees of
freedom associated with a length scale $1$, i.e., the scale of the
lattice, then those on length scales $2, 4, \ldots, 2^{-h}, \ldots$,
with $h<0$. The output of the integration of the first $|h|$ steps of
this iterative procedure can be written as a Grassmann integral
similar to the original one, with the important difference that the
`bare potential' $V(\psi,A)+(\psi,\phi)$ is replaced by an effective
one, $V^{(h)}(\psi,A,\phi)$, that, after appropriate rescaling,
converges to a non-trivial infrared fixed point as $h\to-\infty$. The
large-distance asymptotics of the correlation functions of the dimer
model can thus be computed in terms of those of such an infrared
fixed-point theory, or of those of any other model with the same fixed
point (i.e., of any other model in the same \textit{universality
  class}, the \textit{Luttinger} universality class).  The reference
model we choose for this asymptotic comparison is described in
\cite[Section 4]{GMT20}, which we refer the reader to for additional
details. It is very similar to the Luttinger model, and differs from
it just for the choice of the quartic interaction: it describes a
system of Euclidean chiral fermions in $\mathbb R^2$ (modeled by
Grassmann fields denoted $\psi^\pm_{x,\omega}$, with $x\in\mathbb R^2$
the space label and $\omega\in\{+,-\}$ the chirality label), with
relativistic propagator and a non-local (in both space dimensions,
contrary to the case of the Luttinger model) density-density
interaction\footnote{\label{foot.mass}By `density' of fermions with
  chirality $\omega$ we mean the quadratic monomial
  $\psi^+_{x,\omega}\psi^-_{x,\omega}$; the reference model we
  consider has an interaction coupling the density of fermions with
  chirality $+$ with that of fermions with opposite chirality, see
  \cite[Eq.(4.11)]{GMT20}. For later reference, we also introduce the
  notion of fermionic `mass' of chirality $\omega$, associated with
  the off-diagonal (in the chirality index) quadratic monomial
  $\psi^+_{x,\omega}\psi^-_{x,-\omega}$.}.
The bare parameters of the reference model, in particular the strength of its density-density interaction, are chosen in such a way that its infrared fixed point coincides with the one of our dimer model of interest. 
The remarkable feature of the reference model is that, contrary to our dimer model, it is exactly solvable in a very strong sense: its correlation functions can all be computed in closed form. For our purposes, the relevant correlations are 
those denoted\footnote{The label $R$ stands for `reference' or `relativistic'.} $G^{(2,1)}_{R,\omega,\omega'}$ 
(the vertex function of the reference model, corresponding to the correlation of the density of chirality $\omega$ with a pair of Grassmann fields of chirality $\omega'$), 
$G^{(2)}_{R,\omega}$
(the interacting propagator, corresponding to the correlation between two Grassmann fields of chirality $\omega$), $S^{(1,1)}_{R,\omega,\omega}$ (the density-density correlation between two densities with the same chirality $\omega$) and $S^{(2,2)}_{R,\omega,-\omega}$ (the mass-mass correlation between two masses -- see footnote \ref{foot.mass} -- of opposite chiralities): these are the correlations, in terms of which  the asymptotics of the 
vertex function, interacting propagator and dimer-dimer correlation of our dimer model can be expressed.

\begin{remark}
The connection between the interacting propagator of the dimer model
and that of the reference model can be read from \eqref{h10ab};
similarly, the one between the vertex functions of the two models can
be read from \eqref{h10a}. Moreover, in view of the asymptotics of
$S^{(1,1)}_{R,\omega,\omega}$ and of $S^{(2,2)}_{R,\omega,-\omega}$, see
\cite[Eqs.(4.17) and (4.19)]{GMT20}, \eqref{hh110} can be rewritten as
\begin{equation}\label{rel2p}\sum_{\omega=\pm} \big[K^{(1)}_{\omega,j,\ell}
K^{(1)}_{\omega,j',\ell'}S^{(1,1)}_{R,\omega,\omega}(x,y)+K^{(2)}_{\omega,j,\ell}
K^{(2)}_{-\omega,j',\ell'}S^{(2,2)}_{R,\omega,-\omega}(x,y)e^{2ip^\omega\cdot(x-y)} \big]\end{equation}
plus a faster decaying remainder, which explains the connection between the
dimer-dimer correlation and the density-density and mass-mass
correlations of the reference model.
\end{remark}

The fact that the infrared behavior of the dimer model discussed in
this paper can be described via the same reference model used for the
dimer model in \cite{GMT20} is a priori non-obvious. In fact, the
Grassmann representation of our non-planar dimer model involves
Grassmann fields labelled by $x\in\Lambda$ and
$\ell\in\mathcal I=\{1,\ldots,m^2/2\}$: therefore, one could expect
that the infrared behavior of the system is described in terms of a
reference model involving fields labelled by an index
$\ell\in\mathcal I$. This, a priori, could completely change at a
qualitative level the nature of the infrared behavior of the system,
which crucially depends on the number of mutually interacting massless
fermionic fields. For instance, it is well known that 2D chiral
fermions with an additional spin degree of freedom (which is the case
of interest for describing the infrared behavior of the 1D Hubbard
model), behaves differently, depending on the \textit{sign} of the
density-density interaction: for repulsive interactions it behaves
qualitatively in the same way as the Luttinger model \cite{BFM14},
while for attractive interaction the model dynamically generates a
mass and enters a `Mott-insulator' phase \cite{LiWu}.  In our setting,
remarkably, despite the fact that the number of Grassmann fields used
to effectively describe the model is large for a large elementary cell
(and, in particular, is always larger than 1), the number of massless
fields is the same as in the case of \cite{GMT20}: in fact, out of the
$m^2/2$ fields $\psi^\pm_{(x,\ell)}$ with $\ell\in\{1,\ldots,m^2/2\}$,
all but one of them are \textit{massive}, i.e., their correlations
decay exponentially to zero at large distances, with rate proportional
to the inverse lattice scaling ({this is a direct consequence of the fact that, as proven in \cite{KOS}, the characteristic polynomial $\mu$ has only two zeros}).  Therefore, for the purpose
of computing the generating function, we can integrate out the massive fields in one
single step of the iterative integration procedure, after which we are
left with an effective theory of a single massless Grassmann field
with chirality index $\omega$ associated with the two zeros of $\mu$,
see \eqref{eq:mu}, completely analogous to the one studied in
\cite[Section 6]{GMT20}. See the next section for details.

\medskip

While the proof of the fine asymptotic result summarized in
Proposition \ref{asimptrelprpgtr} is hard, and based on the
sophisticated procedure just described, the proof of Theorem
\ref{mainthrm} \textit{given} Proposition \ref{asimptrelprpgtr} is
relatively easy, and close to the analogous proof discussed in
\cite[Section 5]{GMT20}. We provide it here. Let us start with one definition. Given the  face
$\eta_0\in \bar F$ ($\bar F$ and $\eta_x,x=(x_1,x_2)$ were defined in Section \ref{sec:results}, just before Theorem \ref{mainthrm}),
let $\mathcal E_{1,0}$ (resp. $\mathcal E_{2,0}$), be the set of vertical (resp. horizontal) edges crossed
by the horizontal (resp. vertical) path $C_{\eta_0\to\eta'}$ connecting $\eta_0$ to the face $\eta'\in\bar F$ given by $\eta'=\eta_{(1,0)}$  (resp. $\eta'=\eta_{(0,1)}$). See Fig.\ref{fig:9},
where the same paths and edge sets around the cell $B_x$ rather than $B_0$ are shown. 
For $e\in \mathcal E_{q,0}$, $q=1,2$, we let $(x(e),\ell(e))$ denote the coordinates of its black vertex  and $j(e)\in \mathcal J_{\ell(e)}$ the type of the edge.  We also recall from Section \ref{sec:results} that $\sigma_e=\pm1$
is defined  in \eqref{heightfunction}.
\begin{proposition}
\label{th:3}
For $q=1,2$ and $\omega=\pm$, one has
\begin{equation}\label{haldkadscaling}
\sum_{e \in \mathcal E_{q,0}}\sigma_e \frac{K^{(1)}_{\omega,j(e),\ell(e)}}{Z\sqrt{1-\tau^2}}=-i\omega\sqrt{\nu}\, {\mathrm{d}}_q {\phi}_\omega
\end{equation}
where $\nu=(1-\tau)/(1+\tau)$, and 
\begin{eqnarray}
  \label{eq:D12}\begin{split}
& {\mathrm{d}}_1 \phi_\omega := \phi_\omega(x + (1,0)))-\phi_\omega(x)=\omega\beta_\omega,\\
& {\mathrm{d}}_2 \phi_\omega := \phi_\omega(x + (0,1)))-\phi_\omega(x)=-\omega\alpha_\omega.  \end{split}
\end{eqnarray}
\end{proposition}
\begin{proof}
  Start with the Ward Identity in Fourier space \eqref{FspWid}
  evaluated for $k$ replaced by $k+ p^\omega$ and substitute
  \eqref{h10ab} and \eqref{h10a} in it for $\mathfrak{c} \to
  0$. Recalling that $0<\mathfrak{c}<|k|,|p|,|k+p|<2\mathfrak{c}$ we obtain for $\mathfrak c$ small
\begin{align}
\sum_{\omega'=\pm} \mathcal{D}_{\omega'}(p)G^{(2,1)}_{R,\omega',\omega}(k,p)=(G^{(2)}_{R,\omega}(k)-G^{(2)}_{R,\omega}(k+p))(1+O(\mathfrak{c}^{1/2}))
\end{align}
where $\mathcal{D}_{\omega}(p):=-i\sum_{e \in \mathcal{E}}
K^{(1)}_{\omega,j(e),\ell(e)}p \cdot v(e)$, with $\mathcal E=\mathcal E_{1,0}\cup \mathcal E_{2,0}$ the set of edges defined in Proposition \ref{prop:WI}.
Now comparing the above relation
with the identity \eqref{h11}, by using \eqref{symmrefprop} and by
identifying terms at dominant order for  $|p|$ small we obtain (recall the definition of $D_\omega(p)$ right before \eqref{h11}):
\begin{align}
\mathcal{D}_\omega(p) D_{-\omega}(p)+\tau \mathcal{D}_{-\omega}(p)  D_{\omega}(p)=Z(1-\tau^2)D_{\omega}(p) D_{-\omega}(p).
\end{align}
Letting $p=(p_1,p_2), v(e)=(v_1(e),v_2(e))$, 
imposing $p_2=0,p_1 \neq 0$ first and $p_1=0,p_2 \neq 0$ then, we find a linear system for the coefficients $-i\sum_{e \in \mathcal{E}} K_{\omega,\ell(e),j(e)} v_q(e)$, for $q=1,2$ and $\omega=\pm$ whose solution is
\begin{equation}\label{sollinsyst}\begin{split}
& \sum_{e \in \mathcal{E}} K^{(1)}_{\omega,j(e),\ell(e)} v_1(e)=iZ(1-\tau) \alpha_\omega,\\
& \sum_{e \in \mathcal{E}}  K^{(1)}_{\omega,j(e),\ell(e)} v_2(e)=iZ(1-\tau)\beta_\omega.\end{split}
\end{equation}
Note that, by the very definition of $\mathcal E=\mathcal E_{1,0}\cup \mathcal E_{2,0}$, if $e\in \mathcal E$, then $v_1(e)\neq 0$ iff
$e\in \mathcal E_{2,0}$, while $v_2(e)\ne 0 $ iff $e\in \mathcal E_{1,0}$. 
Recall also the 
relation between $v(e)$ and $\sigma_e$ outlined in Remark
\ref{relvsigma}: in view of this, \eqref{sollinsyst} is equivalent to 
\begin{align}
\sum_{e \in \mathcal E_{1,0}} \frac{K^{(1)}_{\omega,j(e),\ell(e)}}{Z\sqrt{1-\tau^2}} \sigma_e=-i\sqrt{\frac{1-\tau}{1+\tau}}\,\beta_\omega= -i\omega \sqrt{\nu}\, {\mathrm{d}}_1 \phi_\omega \\
  \sum_{e \in \mathcal E_{2,0}}  \frac{K_{\omega,j(e),\ell(e)}}{Z\sqrt{1-\tau^2}} \sigma_e=i\sqrt{\frac{1-\tau}{1+\tau}}\, \alpha_\omega= -i\omega\sqrt{\nu}\,{\mathrm{d}}_2  \phi_\omega, 
\end{align}
where we used $\nu=(1-\tau)/(1+\tau)$ and the definition \eqref{eq:D12}.
\end{proof}

\begin{proof}[Proof of Theorem \ref{mainthrm}]

Given Proposition \ref{th:3}, the proof of Theorem \ref{mainthrm} is
essentially identical to that of \cite[Eq.(2.47)]{GMT20} and of \cite[Proof of (7.26)]{GMTaihp}. Here we give only a sketch and we emphasize only
the role played by the relation \eqref{haldkadscaling} that we have just proven.

First of all, we choose a 
path $C_{\eta_{x^{(1)}}\to\eta_{x^{(2)}}}$
from face $\eta_{x^{(1)}}$ to $\eta_{x^{(2)}}$ that crosses only
edges that join different cells. Since
$\eta_{x^{(1)}},\eta_{x^{(2)}}\in \bar F$, the path
$C_{\eta_{x^{(1)}}\to\eta_{x^{(2)}}}$ visits a 
sequence of faces $\eta_{y^{(1)}},\dots,\eta_{y^{(k)}}\in \bar F$, with $y^{(1)}={x^{(1)}}, y^{(k)}={x^{(2)}}$ and $|y^{(a)}-y^{(a+1)}|=1$. The set of 
 edges crossed by
the path between $\eta_{y^{(a)}}$ and $\eta_{y^{(a+1)}}$, denoted $\mathcal E_{(a)}$, is a translation of either $\mathcal E_{1,0}$ (if $y^{(a+1)}-y^{(a)}$ is horizontal) or $\mathcal E_{2,0}$
(if $y^{(a+1)}-y^{(a)}$ is vertical).
 Similarly, one defines a path
 $C_{\eta_{x^{(3)}}\to\eta_{x^{(4)}}}$ and correspondingly a sequence of faces $\eta_{z^{(1)}},\dots,\eta_{z^{(k')}}\in \bar F$
 and $\mathcal E'_{(a)}$ the set of edges crossed by the path between $\eta_{z^{(a)}}$ and $\eta_{z^{(a+1)}}$.
The two paths can be chosen so that $C_{\eta_{x^{(1)}}\to\eta_{x^{(2)}}}$ is of length $O(|x^{(1)}-x^{(2)}|)$ and  $C_{\eta_{x^{(3)}}\to\eta_{x^{(4)}}}$  is of length $O(|x^{(3)}-x^{(4)}|)$, while they are at mutual distance at least of order $\min(|x^{(i)}-x^{(j)}|,i\ne j)$. See \cite{GMTaihp} for more details.
 
From the definition \eqref{heightfunction} of height function, we see that
\begin{equation}\label{heightcovar}
\mathbb E_\l\left[(h(\eta_{x^{(1)}})-h(\eta_{x^{(2)}}));(h(\eta_{x^{(3)}})-h(\eta_{x^{(4)}}))
\right]
=\sum_{\substack{1\le a< k,\\ 1\le  a'< k'}} 
\sum_{\substack{e\in \mathcal E_{(a)}, \\ e'\in \mathcal E'_{(a')}}} \sigma_e \sigma_{e'} \mathbb{E}_\lambda[\openone_e;\openone_{e'}].
\end{equation}
As a consequence of Proposition \ref{asimptrelprpgtr}, for edges $e,e'$ with black sites of coordinates $(x,\ell),(x',\ell')$ and with orientations $j,j'$, respectively, we have that 
\begin{equation}
  \label{eq:dimdim}\begin{split}
\mathbb{E}_\lambda[\openone_e;\openone_{e'}]&= \sum_{\omega=\pm} \frac{ K^{(1)}_{\omega,j,\ell}}{Z\sqrt{1-\tau^2}}\frac{ K^{(1)}_{\omega,j',\ell'}}{Z\sqrt{1-\tau^2}}\frac1{4\pi^2 ( \phi_\omega(x-x'))^2}\\
&+ \frac{B}{4\pi^2}\sum_{\omega=\pm} \frac{K^{(2)}_{\omega,j,\ell}K^{(2)}_{-\omega,j',\ell'}}{ | \phi_\omega(x-x')|^{2(1-\tau)/(1+\tau)}}e^{2ip^\omega (x-x')}+R_{j,j',\ell,\ell'}(x,x').
\end{split}\end{equation}
At this point we plug this expression into \eqref{heightcovar}. The
oscillating term in \eqref{eq:dimdim}, proportional to $B$, and the
error term $R_{j,\ell,j',\ell}(x,x')$, once summed over $e,e'$,
altogether end up in the error term in \eqref{eq:35l} (see the
analogous argument in \cite[Section 3.2 and 7.3]{GMTaihp}).  As for
the main term involving $K^{(1)}_{\omega,j,\ell}$, we observe that if
we fix $a,a'$, then for
$e \in \mathcal E_{(a)},e'\in \mathcal E'_{(a')}$ we can replace in
\eqref{eq:dimdim} $\phi_\omega(x-x')$ by
$\phi_\omega(y^{(a)}-z^{(a')})$, up to an error term of the same order
as $R_{j,j',\ell,\ell'}(x,x')$, which again contributes to the error
term in \eqref{eq:35l}.  We are thus left with
\begin{equation}\begin{split}
&  \sum_{\omega=\pm}\sum_{\substack{1\le a< k,\\ 1\le a'< k'}}\frac1{4\pi^2 ( \phi_\omega(y^{(a)}-z^{(a')}))^2} \sum_{e\in \mathcal E_{(a)}} \sigma_e \frac{ K_{\omega,j(e),\ell(e)}}{Z\sqrt{1-\tau^2}}\sum_{e'\in \mathcal E'_{(a')}}\sigma_{e'}   \frac{ K_{\omega,j(e'),\ell(e')}}{Z\sqrt{1-\tau^2}}\\
 & =- \nu\sum_{\omega=\pm}\sum_{\substack{1\le a< k,\\ 1\le a'< k'}}\frac{(y^{(a+1)}-y^{(a)})\cdot{\rm d}\phi_\omega\, (z^{(a'+1)}-z^{(a')})\cdot{\rm d}\phi_\omega} {4\pi^2 ( \phi_\omega(y^{(a)}-z^{(a')}))^2}\\
 & =- \frac{\nu}{2\pi^2}\Re \Big[\sum_{\substack{1\le a< k,\\ 1\le a'< k'}}\frac{(y^{(a+1)}-y^{(a)})\cdot{\rm d}\phi_+\, (z^{(a'+1)}-z^{(a')})\cdot{\rm d}\phi_+} { ( \phi_+(y^{(a)}-z^{(a')}))^2}
    \Big]
\end{split}\end{equation}
where in the first step we used Proposition \ref{th:3} and defined ${\rm d}\phi_\omega:=({\rm d}_1\phi_\omega, {\rm d}_2\phi_\omega)$. 
As explained in \cite[Section 5.2]{GMT20} and \cite[Section 7.3]{GMTaihp} (see also \cite[Section 4.4.1]{KOS} in the non-interacting case), this sum equals the integral in the complex plane
\begin{eqnarray}
  - \frac{\nu}{2\pi^2}\Re\int_{\phi_+(x^{(1)})}^{\phi_+(x^{(2)})}dz\int_{\phi_+(x^{(3)})}^{\phi_+(x^{(4)})}dz'\frac1{(z-z')^2}
\end{eqnarray}
(which equals the main term in the r.h.s. of \eqref{eq:35l}), plus an error term (coming from the Riemann approximation)  estimated as in the r.h.s. of \eqref{eq:35l}.
\end{proof}

\section{Proof of Proposition \ref{asimptrelprpgtr}}\label{lasezionepesante}
In this section we give the proof of Proposition
\ref{asimptrelprpgtr} (which immediately implies Theorem \ref{thm:1}, as already commented above), 
via the strategy sketched after its
statement. As explained there, the novelty compared to the proof in
\cite[Section 6]{GMT20} is the reduction to an effective model involving
a \textit{single} Grassmann critical field $\varphi$, of the same form
as the one analyzed in \cite[Section 6]{GMT20}. Therefore, most of this
section will be devoted to the proof of such reduction, which consists
of the following steps. Our starting point is the generating function
of correlations in its Grassmann form, see \eqref{gener}. In \eqref{gener}, we first integrate out the
`ultraviolet' degrees of freedom at the lattice scale, see
Section \ref{sec:5.1} below; the resulting effective theory can be
conveniently formulated in terms of a collection of chiral fields
$\{\psi^\pm_{x,\omega}\}_{x\in \Lambda}^{\omega\in\{+,-\}}$, where
$\psi^\pm_{x,\omega}$ are Grassmann vectors with $|\mathcal I|$
components, which represent fluctuation fields supported in momentum
space close to the unperturbed Fermi points $p_0^\omega$.  Next, we
perform a `rigid rotation' of these Grassmann vectors via a matrix $B$
that is independent of $x$ but may depend on the chirality index
$\omega$; the rotation is chosen so to block-diagonalize the reference
quadratic part of the effective action, in such a way that the
corresponding covariance is the direct sum of two terms, a
one-dimensional one, which is singular at $p^\omega_0$, and a
non-singular one, of dimension $|\mathcal I|-1$; the components
associated with this non-singular $(|\mathcal I|-1)\times(|\mathcal
I|-1)$ block are referred to as the `massive components', which can be
easily integrated out in one step, see Section \ref{sec:5.2} below (this
is the main novel contribution of this section, compared with the
multiscale analysis in \cite{GMT20}).  In Section \ref{sec:5.3} below we
reduce essentially to the setting of \cite{GMT20}, that is, to an
effective theory that involves one single-component ``quasi-particle'' chiral massless
field, which can be analyzed along the same lines as
\cite[Section 6]{GMT20}. Finally, in Section \ref{sec:5.4} we conclude the proof of
Proposition \ref{asimptrelprpgtr}.

\subsection{Integration of the ultraviolet degrees of freedom}\label{sec:5.1}
We intend to compute the generating function \eqref{gener} with $\theta$ boundary conditions. 
We introduce Grassmann variables in Fourier space via the following transformation: 
\begin{align}\label{fourierfields}
\hat \psi^\pm_{k}:= \sum_{x \in \Lambda} e^{\mp ik x} \psi^\pm_{x}, \qquad \psi^\pm_{x}=\frac{1}{L^2}\sum_{k \in \mathcal{P}(\th)} e^{\pm ik x} \hat \psi^\pm_{k},
\end{align}
where we recall that each $\psi_x^\pm$ and each $\hat \psi_k^\pm$ has
$|\mathcal I|$ components and indeed we assume that
$\psi_x^+=(\psi^+_{x,1},\dots,\psi^+_{x,|\mathcal{I}|})$ is a row
vector while similarly $\psi^-_x$ is a column vector (whenever
unnecessary, we shall drop the `color' index $\ell\in\mathcal I$); in
this way the transformation above is performed component-wise.

 For
each $\theta \in\{-1,+1\}^2$, we let $p^\omega_\theta$, $\omega=\pm1$
denote the element of $\mathcal P(\theta)$ that is closest to
$p^\omega_0$\footnote{In the case of more than one momentum at minimum
  distance, any choice of $p_\theta^\pm$ will work. The dependence on
  $L$ of $p_\theta^\pm$ is understood\label{foot:varii}}, we 
  rewrite
 \[\psi^\pm_x=\psi'^{\pm}_{x}+ \Psi^{\pm}_{x}\] with
 \begin{equation}\label{eq:500} \psi'^\pm_{x}=\frac{1}{L^2}\sum_{k \not\in\{p^{+}_\theta,p^-_\theta\}} e^{\pm ik x} \hat \psi^\pm_{k},\quad \Psi^\pm_{x}={\frac1{L^2}}\sum_{k\in\{p^{+}_\theta,p^-_\theta\}}e^{\pm ik x} \hat \psi^\pm_{k}.
 \end{equation}
 Noting that
\[
S_\theta(\psi)=S_\theta(\Psi)+S_\theta(\psi'):=
-\frac1{L^{2}}\sum_{k\in\{p^{+}_\theta,p^-_\theta\}}\hat\psi^+_k
M(k)\hat\psi_k^-- \frac1{L^{2}}\sum_{k\not\in\{p^{+}_\theta,p^-_\theta\}}\hat\psi'^+_k M(k)\hat\psi_k'^-,\] we rewrite
\begin{equation} \label{eq:501}e^{\mathcal{W}_{L}^{(\th)}(A,\phi)}=\Big(\prod_{k\not\in\{p^{+}_\theta,p^-_\theta\}}\mu(k)\Big)\int D\Psi\, e^{S_\theta(\Psi)}\int P(D\psi')\, e^{V(\psi,A)+(\psi,\phi)},
\end{equation}
where $D\Psi=\prod_{k\in\{p^{+}_\theta,p^-_\theta\}}\big(L^{2|\mathcal I|}D\hat\Psi_k\big)$ and the
Grassmann ``measure'' $D\hat\Psi_k$ is defined, as usual, so that
\[\int \Big(\prod_{k\in\{p^{+}_\theta,p^-_\theta\}}D\hat\Psi_k\Big)\, \Big(\prod_{k\in\{p^{+}_\theta,p^-_\theta\}}
\prod_{\ell\in
  \mathcal{I}}\hat\Psi^-_{k,\ell}\hat\Psi^+_{k,\ell}\Big)=1,\] while
we have $\int \Big(\prod_{k\in I}D\hat\Psi_k\Big) Q(\Psi)=0$
whenever $Q(\Psi)$ is a monomial in $\{\hat \Psi^\pm_{k,\ell}\}_{k\in\{p^{+}_\theta,p^-_\theta\},\ell\in\mathcal I}$
of degree strictly lower or strictly larger than $4|\mathcal{I}|$. Moreover, $P(D\psi')$ is
the Grassmann Gaussian integration, normalized so that $\int
P(D\psi')=1$, associated with the propagator
\begin{equation}
g'(x,y)=\int P(D \psi') \psi'^-_{x} \psi'^+_{y}=L^{-2}\sum_{k\not \in\{p^{+}_\theta,p^-_\theta\}} e^{-ik(x-y)} (M(k))^{-1}.
\end{equation}
Note that, since $\psi'^\pm_{x}$ is a vector with $|\mathcal I|$
components, $g'(x,y)$ is an $|\mathcal I|\times
|\mathcal I|$ matrix, for fixed $x,y$.
\begin{remark}
We emphasize also that, since the zeros of $\mu$ are simple, $\mu(k)\ne0$ for every $k\not\in\{p^+_\theta,p^-_\theta\}$ (this is the reason why we singled out the two momenta $p^\o_\theta $ where $\mu$ possibly vanishes and $M$ is not invertible).
\end{remark}

Next we introduce the following
  \begin{definition}
    \label{def:chi}
We let  $\chi_\o: \mathbb R^2
  \longrightarrow [0,1],\o=\pm1$  be two  $C^\infty$ functions in the Gevrey class of order 2, see
  \cite[App.C]{GMTaihp}, with the properties that:
\begin{enumerate}
\item [(i)]
  $\chi_\o(k)=\chi_{-\o}(-k)$,
\item[(ii)]
  $\chi_\o(k)=1$ if $|k-p^\o_0|\le c_0/2$, and $\chi_\o(k)=1$ if $|k-p^\o_0|> c_0$, with $c_0$ a small enough positive constant, such
  that in particular the support of $\chi_+$ is disjoint
  from the support of $\chi_{-}$.
\end{enumerate}
\end{definition}
  We will specify later a more explicit
definition of $\chi_\o$.
We rewrite
$g'=g^{(0)}+g^{(1)}$, with 
\begin{equation}\label{eq:5.5}
\begin{split}
& g^{(0)}(x,y)=L^{-2} \sum_{\omega=\pm} \sum_{k \not\in \{p^+_\theta,p^-_\theta\}} e^{-ik(x-y)}\chi_\o(k) (M(k))^{-1},\\
& g^{(1)}(x,y)=L^{-2} \sum_{k \in \mathcal P(\theta)} e^{-ik(x-y)}(1-\chi_+(k)-\chi_-(k)) (M(k))^{-1}.
\end{split}
\end{equation}  
Since the cutoff functions $\chi_\o$ are  Gevrey functions of order 2, the propagator $g^{(1)}$ has stretched-exponential decay at large distances
\begin{equation}\label{eq:502}
  \|g^{(1)}(x,y)\|\le C e^{-\kappa\sqrt{|x-y|}},\end{equation}
for suitable $L$-independent constants $C,\kappa>0$, cf. with \cite[Eq.~(6.21)]{GMT20} $|\mathcal I|\times
|\mathcal I|$ (recall that the propagators are $|\mathcal I|\times
|\mathcal I|$ matrices; the norm in the l.h.s. is any matrix norm).
In \eqref{eq:502}, $|x-y|$ denotes the graph distance between $x$ and $y$ on  $G_L$.

Using the addition principle for Grassmann Gaussian integrations
\cite[Proposition 1]{GMTaihp}, we rewrite \eqref{eq:501} as
\begin{multline}\label{eq:504}
 e^{\mathcal{W}_{L}^{(\th)}(A,\phi)}= \Big(\prod_{k\not\in \{p^+_\theta,p^-_\theta\}}\mu(k)\Big)\int D\Psi\, e^{S_\theta(\Psi)}\int P_{(0)}(D\psi^{(0)})  \\ \times  \int P_{(1)}(D\psi^{(1)})\, e^{V(\Psi+\psi^{(0)}+\psi^{(1)},A)+(\Psi+\psi^{(0)}+\psi^{(1)},\phi)}\\
 = \Big(\prod_{k \not\in \{p^+_\theta,p^-_\theta\}} \mu(k) \Big) e^{L^2 E^{(0)}+S^{(0)}(J,\phi)} \int D \Psi\, e^{S_\theta(\Psi)}\int P_{(0)} (D\psi^{(0)})\,e^{V^{(0)}(\Psi+\psi^{(0)},J,\phi)}
\end{multline}
where: $P_{(0)}$ and $P_{(1)}$ are the Grassmann Gaussian integrations
with propagators $g^{(0)}$ and $g^{(1)}$, respectively, i.e.,
letting $O_\omega=\{k \in \mathcal
  P(\theta)\setminus\{p^+_\theta,p^-_\theta\}: \chi_\o(k) \neq 0$, 
 \begin{equation} P_{(0)}(D\psi)=\prod_\omega\frac{\Big(L^{2|\mathcal
      I||O_\omega|}\prod_{k\in O_\omega}D\hat\psi_k\Big)
    \exp{\Big(-L^{-2} \sum_{ k \in O_\omega} (\chi_\o(k))^{-1} \hat
      \psi^{+}_k M(k)\hat \psi^-_k }\Big)}{\Big(\prod_{k \in O_\omega}
    \mu(k) (\chi_\o(k))^{-|\mathcal{I}|}\Big)},\label{eq:5.8}
\end{equation}
and a similar explicit expression for $P_{(1)}$ holds; $J=\{J_e\}_{e\in E_L}$ with $J_e=e^{A_e}-1$; $E^{(0)}$, $S^{(0)}$ and $V^{(0)}$ are defined 
via 
\begin{equation}L^2 E^{(0)}+S^{(0)}(J,\phi)+V^{(0)}(\psi,J,\phi)=\log \int P_{(1)}(D\psi^{(1)}) e^{V(\psi+\psi^{(1)},A)+(\psi+\psi^{(1)},\phi)},\label{eq:5.9}\end{equation}
with $E^{(0)},S^{(0)}$ fixed uniquely by the condition that
$V^{(0)}(0,J,\phi)=S^{(0)}(0,0)=0$. Proceeding as in the proof of
\cite[Eq.(6.24)]{GMT20}, one finds that the effective potential
$V^{(0)}$ can be represented as follows:
\begin{equation}\label{eq:503}
V^{(0)}(\psi,J,\phi)=\sum_{\substack{n>0 \\ m,q \geq 0 \\ n+q  \in 2\mathbb{N}}}  \sum^*_{\substack{\underline{x},\underline{y},\un{z} \\ \un\ell,\un \ell',\un \ell',\\ \un s, \un \sigma,\un \sigma' }} \psi^{\un \sigma}_{\un{x},\un{\ell}} J_{\un y,\un \ell', \un s} \phi^{\un \sigma'}_{\un z,\un \ell''} W_{n,m,q;\boldsymbol{a}}(\underline{x},\underline{y},\un{z})
\end{equation} 
where the second sum runs over $\un{x} \in \Lambda^n, \un{y} \in
\Lambda^m,\un{z} \in \Lambda^q$, $\un{\ell} \in
\mathcal{I}^n,\un{\ell}' \in \mathcal{I}^m,\un{\ell}'' \in
\mathcal{I}^q$, $ \un{s} \in
\mathcal{J}_{\ell_1}\times\dots\times\mathcal{J}_{\ell_m}$,
$\un \sigma \in \{+,-\}^n,\un \sigma' \in \{+,-\}^q$ (the $*$ on the
sum indicates the constraint that $\sum_{i=1}^n \sigma_i+\sum_{i=1}^q
\sigma_i'=0$), and we defined
$J_{\un{y},\un{\ell'},\un{s}}:=\prod_{i=1}^m J_{y_i,\ell'_i,s_i}$
(here $J_{y,\ell,s}$ stands for $J_e$ when the edge $e \in E_L$ has
black site of coordinates $(y,\ell)$ and orientation $s \in
\mathcal{J}_{\ell}$), $\psi^{\un \sigma}_{\un x,\un
  \ell}:=\prod_{i=1}^n \psi^{\sigma_i}_{x_i,\ell_i}$, and similarly
for $\phi^{\un \sigma'}_{\un z,\un \ell''}$; finally, ${\boldsymbol a}:=(\underline\ell,\underline\sigma,\underline\ell',\underline s,\underline\ell'',\underline\sigma')$. Without loss of
generality, we can assume that the kernels $W_{n,m,q;(\un \ell,\un\sigma,\un
  \ell',\un s,\un \ell'',\un\sigma')}$ are symmetric under permutations of
the indices $(\un y,\un \ell', \un s)$ and antisymmetric both under
permutations of $(\un x,\un \ell,\un \sigma)$ and of $(\un z, \un \ell'',\un\sigma')$. A
 representation similar to \eqref{eq:503} holds also for
$S^{(0)}(J,\varphi)$ with kernels $W^{0,m,q}_{\boldsymbol a}(\un y, \un z)$, where ${\boldsymbol{a}}=(\un \ell',\un s,\un \ell'',\un\sigma')$. As discussed after \cite[Eq.(6.27)]{GMT20}, using
the Battle-Brydges-Federbush-Kennedy determinant formula and the
Gram-Hadamard bound \cite[Sec. 4.2]{GM01} one finds that $E^{(0)}$ and
the values of the kernels $W_{n,m,q;\boldsymbol a}(\un x,\un y,\un z)$ at fixed positions $\un
x,\un y,\un z$ are real analytic functions of the parameter $\lambda$,
for $|\lambda|\le \lambda_0$ and $\lambda_0$ sufficiently small but independent of $L$. Moreover, in the analyticity
domain, $|E^{(0)}| \leq C |\lambda|$, and
\begin{align}\label{kernorm}
\| W_{n,m,q}
\|_{\kappa,0} \leq C^{n+m+q} |\lambda|^{\mathbbm{1}_{n+q>2} \max\{1,c(n+q)\}}
\end{align}
for suitable positive constants $C,c$ independent of $L$. Here the weighted norm $\| \cdot \|_{\kappa,0}$ is defined as
\begin{equation}
 \| W_{n,m,q}
 \| _{\kappa,0}:=L^{-2} \sup_{\boldsymbol{a}} \sum_{\underline{x}, \underline{y}, \underline{z}} |W_{n,m,q;\boldsymbol{a}}(\underline{x},\underline{y},\un z)| e^{\frac{\kappa}{2}\sqrt{\delta(\underline{x},\underline{y}, \un z)}},
\end{equation} where $\kappa>0$ is the same as in \eqref{eq:502}, and $\delta(\cdot)$ denotes the tree distance, that is the length of the shortest tree on the torus connecting points with the given coordinates. 

\begin{remark} The kernels of the effective potential $V^{(0)}$, of $S^{(0)}$, as well as the constant $E^{(0)}$, depend on $\theta$, 
because both the interaction $V(\psi,A)$ in \eqref{eq:5.9} and the
propagator $g^{(1)}$ involved in the integration do. Both these
effects can be thought of as being associated with boundary conditions
assigned to the Grassmann fields,
{periodic in both coordinate
directions for $\theta=(-,-)$, anti-periodic in both coordinate
directions for $\theta=(+,+)$}, and mixed (periodic in one direction
and anti-periodic in the other) in the remaining two cases. Therefore,
using Poisson summation formula (see e.g. \cite[App. A.2]{GMTaihp}, where notations are different), both $g^{(1)}$ and the kernels of
$V^{(0)}$ and $S^{(0)}$ can be expressed via an `image rule',
analogous to the summation over images in electrostatics, of the
following form:
\begin{equation} g^{(1)}(x,y)=\sum_{n=(n_1,n_2)\in\mathbb Z^2}(-1)^{\frac{\theta_1+1}2 n_1+\frac{\theta_2+1}2 n_2}g^{(1),\infty}(x-y+nL),\end{equation}
where $g^{(1),\infty}(x)=\lim_{L\to\infty}g^{(1)}(x,0)$ (an analogous sum rule holds for the kernels of $V^{(0)}$ and $S^{(0)}$). From this
representation, together with the decay bounds mentioned above on $g^{(1)}$ and on the
kernels of the effective potential, it readily follows that the
dependence upon $\theta$ of these functions is a finite-size effect that is stretched-exponentially small in $L$. Similarly, the dependence
upon $\theta$ of $E^{(0)}$ corresponds to a {stretched-exponentially} small
correction as $L\to\infty$ (see also \cite[Appendix
  A.2]{GMTaihp}). Therefore, all these corrections are irrelevant for
the purpose of computing the thermodynamic limit of thermodynamic
functions and correlations. For this reason and for ease of notation,
here and below we will not indicate the dependence upon $\theta$
explicitly in most of the functions and constants involved in the
multiscale construction.
\end{remark}

\subsection{Integration of the massive degrees of freedom}
\label{sec:5.2}
Using \eqref{eq:5.8} in \eqref{eq:504} and renaming $\Psi+\psi^{(0)}\equiv \psi$, we get 
\begin{multline}\label{eq:507}
e^{\mathcal{W}_{L}^{(\th)}(A,\phi)}=e^{L^2(t^{(0)}+ E^{(0)})+S^{(0)}(J,\phi)} \\ 
 \times  \int D\psi\, e^{-L^{-2} \sum_\omega\sum_{k \in \mathcal{B}_\omega} (\chi_\o(k))^{-1}\hat \psi^+_{k} M(k) \hat \psi^-_{k}}e^{V^{(0)}(\psi,J,\phi) }
\end{multline}
where, recalling that $O_\o$ was defined right before \eqref{eq:5.8},
\[\mathcal{B}_\omega:=O_\omega\cup\{p^\omega_\theta\}=\{k \in
\mathcal{P}(\theta): \chi_\o(k)\ne0 \},\] 
$D\psi:=\prod_{\omega=\pm}\Big(L^{2|\mathcal
  I||\mathcal{B}_\omega|}\prod_{k\in
  \mathcal{B}_\omega}D\hat\psi_k\Big)$ and
we have set
$$t^{(0)}:=\frac1{L^2}\sum_{k \in (\cup_\omega\mathcal B_\omega)^c} \log \mu(k)+\frac{|\mathcal{I}|}{L^2}\sum_{\omega}\sum_{k \in O_\omega}  \log \chi_\o(k).$$

Since $p^+_0$ is a simple zero of $\mu(k)$, there exists an invertible complex matrix $B_+$ such that \begin{equation}
B_+ M(p^+_0) B_+^{-1}=\begin{pmatrix}
0 & 0 \\
0 & A_+ 
\end{pmatrix} \label{A+}
\end{equation} for an invertible $(|\mathcal{I}|-1) \times (|\mathcal{I}|-1) $ matrix $A_+$. Clearly, $B_+$ (and, therefore, $A_+$) is not defined uniquely; we choose it arbitrarily, in such a way that \eqref{A+} holds, 
and fix it once and for all. Taking the complex conjugate in the above equation and using the symmetry of $M$, see \eqref{fermippvel}, one finds that the same relation holds at $p^-_0$ with matrices $B_-:=\overline{B_+}$, $A_-:=\overline{A_+}$. Let  ${\bf M}_\omega(k):=B_\omega M(k)B_\omega^{-1}$, and define the matrices 
$T_\omega(k), W_\omega(k), U_\omega(k)$ and $V_\omega(k)$ of sizes $1\times 1$, $(|\mathcal{I}|-1) \times (|\mathcal{I}|-1)$, $1 \times (|\mathcal{I}|-1)$ and $(|\mathcal{I}|-1) \times 1$, respectively, via 
\begin{equation}
\begin{pmatrix}
T_\o (k) & U_\o(k) \\ V_\o(k) & W_\o(k)
\end{pmatrix}:={\bf M}_\o(k).\label{eq:5.15}
\end{equation}
Analyticity of $M(k)$ in $k$ implies, in particular, that $T_\o(k+p^\omega_0)$, $U_\o(k+p_0)$ and $V_\o(k+p^\o_0)$ are all $O(k)$ as $k\to 0$, while $W_\o(k+p^\o_0)=A_\o+O(k)$. Let $\mathcal B^{(2)}_\omega\supset \mathcal B_\o$ be the ball 
centered at $p_0^\omega$ with radius $2c_0$
, and assume that $c_0$
is so small that $\inf_{k\in\mathcal B^{(2)}_\omega}|\det W_\o(k)|$ is positive. 
Taking the determinant at both sides of \eqref{eq:5.15}, 
letting $\rho_\o:=\det A_\o$, we find that
\[\mu(k)\stackrel{k\to p^\omega_0}=\rho_\o T_\o(k)+O((k-p^\o_0)^2)\]  so that, recalling \eqref{eq:alphabeta},
\begin{align}
T_\o(k+p^\o_0)\stackrel{k\to0}=\frac{\alpha^0_\o k_1+ \beta^0_\o k_2}{\rho_\o}+O(k^2). 
\end{align}
Since $W_\o(k)$ is non singular on $\mathcal B_\o$, for $k\in\mathcal
B_\omega$ we can block diagonalize ${\bf M}_\o$ as
\begin{equation}\label{eq:5.17}
{\bf M}_\o(k)=\begin{pmatrix}
1 & U_\o(k)W^{-1}_\o(k) \\
0 & \mathbbm{1}
\end{pmatrix}
\begin{pmatrix} 
{\bf T}_\o(k) & 0 \\ 
0 & W_\o(k)\end{pmatrix}
\begin{pmatrix}
1 & 0 \\ 
W^{-1}_\o(k) V_\o(k) & \mathbbm{1}
\end{pmatrix}
\end{equation}
 where ${\bf T}_\o(k):=T_\o(k)-U_\o(k) W^{-1}_\o(k)V_\o(k)$ is the
 Schur complement of the block $W_\o$. Note that from the properties
 of $U_\o,V_\o,W_\o$, the function ${\bf T}_\o$
 satisfies $${\bf T}_\o(k+p^\o_0)\stackrel{k\to0}=\frac{\alpha^0_\o k_1+
 \beta^0_\o k_2}{\rho_\o}+O(k^2),$$ like $T_\o$. In view of this
 decomposition, we perform the following change of Grassmann
 variables: for $k \in \mathcal{B}_\o$ we define
\begin{equation}\label{eq:5.18}\begin{split}
& (\hat{\varphi}^+_{k},\hat{\xi}^+_{k,1}, \dots,\hat{\xi}^+_{k,|\mathcal{I}|-1}):=\hat{\psi}^+_{k} B_\o^{-1} \begin{pmatrix}
1 & U_\o(k) W^{-1}_\o(k) \\ 0 & \mathbbm{1}
\end{pmatrix} \\ 
& (\hat{\varphi}^-_{k}, \hat{\xi}^-_{k,1}, \dots, \hat{\xi}^-_{k,|\mathcal{I}|-1})^T:= \begin{pmatrix}
1 & 0 \\  W^{-1}_\o(k) V_\o(k) & \mathbbm{1}   
\end{pmatrix} B_\o \hat{\psi}^-_{k}.
\end{split}\end{equation}
For later convenience, we give the following
\begin{lemma}
  \label{lemma:chi}
  Define    $\psi^\pm_x:=L^{-2}\sum_\o \sum_{k \in \mathcal{B}_\o} e^{\pm ikx} \hat\psi^+_k$
  and
  \[\xi^\pm_x(\o):=L^{-2}\sum_{k \in \mathcal{B}_\o}e^{\pm ikx} \hat\xi^+_{k},\quad 
\varphi^\pm_x(\o):=L^{-2}\sum_{k \in \mathcal{B}_\o}e^{\pm ikx} \hat\varphi^+_k.\]
  Then, the inverse of the transformation \eqref{eq:5.18} in $x$ space is
 \begin{equation}\label{trans}
\begin{split} 
& \psi^+_{x,\ell}=\sum_{\omega} \Big( \varphi^+_{x}(\o)(B_\o)_{1 \ell} +(\varphi^+(\o) \ast  \tau^+_{\omega,\ell})_x+\sum_{j=2}^{|\mathcal{I}|} \xi^+_{x,j-1}(\o) (B_\o)_{j \ell} \Big )\\
& \psi^-_{x,\ell}=\sum_{\omega} \Big( (B_\o^{-1})_{\ell 1}\varphi^-_{x}(\o) +( \tau^-_{\omega,\ell}\ast \varphi^-(\o))_x+\sum_{j=2}^{|\mathcal{I}|}(B_\o^{-1})_{\ell j} \xi^-_{x,j-1}(\o)  \Big )
\end{split}
\end{equation} 
where
\begin{equation}\label{taupm}\begin{split} &{\tau}_{\omega,\ell}^+(x):=-L^{-2}\sum_{k \in \mathcal{P}(\theta)} \sum_{j=2}^{|\mathcal{I}|} e^{ikx} \chi_\o\big(\tfrac{k+p_0^\omega}2\big)(U_\o(k)\cdot W^{-1}_\o(k))_j (B_\o)_{j\ell}\\
&{\tau}_{\omega,\ell}^-(x):=-L^{-2}\sum_{k \in \mathcal{P}(\theta)} \sum_{j=2}^{|\mathcal{I}|}e^{-ikx}\chi_\o\big(\tfrac{k+p_0^\omega}2\big)(B^{-1}_\o)_{\ell j}(W_\o^{-1}(k) \cdot V_\o(k))_j .\end{split}
\end{equation}
\end{lemma}
\begin{proof} The proof is essentially an elementary computation (one inverts the linear relation \eqref{eq:5.18} for given $k$ and then takes the Fourier transform to obtain the expression in real space) but there is a slightly delicate point, that is to see where the cut-off function $\chi_\o\big(\tfrac{k+p_0^\omega}2\big)$ comes from. After a few elementary linear algebra manipulations, one finds
  that $\psi^+_{x,\ell}$ equals an expression like in the r.h.s. of
  \eqref{trans}, where the term $(\varphi^+(\o) \ast
  \tau^+_{\omega,\ell})_x$ is replaced by
  \[
  \label{senzachi}
  \frac1{L^2}\sum_\o\sum_{k\in \mathcal B_\o}\hat\varphi^+_k f^\o_\ell(k)e^{i k x},\quad
  f^\o_\ell(k):=-\sum_{j=2}^{|\mathcal I|}\big(U_\o(k)W^{-1}_\o(k)\big)_j(B_\o)_{j\ell}.
  \]
  Since the sum is restricted to $k\in \mathcal B_\o$, we can freely multiply the summand by  $\chi_\o\big(\tfrac{k+p_0^\omega}2\big)$, which is identically equal $1$ there, since the argument is at distance at most $c_0/2$ from $p^\o_0$. At that point, we use the fact that
  \[
\hat\varphi^+_k{\bf 1}_{k\in \mathcal B_\o}=\sum_x\varphi^+_x(\o)e^{-i k x}
\]
and we immediately obtain that \eqref{senzachi} coincides with $(\varphi^+(\o) \ast  \tau^+_{\omega,\ell})_x$, with $\tau^+_{\o,\ell}$ as in \eqref{taupm}.
\end{proof}

At this point we go back to  \eqref{eq:507}, that we rewrite as
\begin{multline}
e^{\mathcal{W}_{L}^{(\th)}(A,\phi)}=e^{L^2 ({\bf t}^{(0)}+ E^{(0)})+S^{(0)}(J,\phi)}\\
\times \int  D\varphi\,e^{-L^{-2} \sum_{\omega}\sum_{k \in \mathcal{B}_\omega} (\chi_\o(k))^{-1}\hat \varphi^+_{k} {\bf T}_\o(k) \hat \varphi^-_{k}} \int P_W(D\xi)\, e^{\widetilde V^{(0)}(\varphi,\xi,J,\phi)}
\end{multline}
where
$${\bf t}^{(0)}:= L^{-2}\sum_{k \in O'}\log \mu(k)+L^{-2}\sum_{\omega=\pm}\sum_{k\in\mathcal B_\omega}\big(\log\det W_\omega(k)+\log \chi_\o(k)\big),$$
$D\varphi:=\prod_{\omega=\pm}\Big(L^{2|\mathcal{B}_\omega|}\prod_{k\in \mathcal{B}_\omega}D\hat\varphi_k\Big)$, $P_{W}(D\xi)$ is the normalized Gaussian Grassmann integration with propagator (which is a $(|\mathcal I|-1)\times (|\mathcal I|-1)$ matrix)
\begin{equation}\label{eq:508}
g^W_{\omega,\omega'}(x,y):=\int P(D\xi)\, \xi^-_x(\o) \xi^+_y(\o')= \frac{\delta_{\o,\o'}}{L^2}\sum_{k \in P(\theta)} e^{-ik(x-y)} \chi_\o(k) (W_\o(k))^{-1},
\end{equation}
and $\widetilde V^{(0)}(\varphi,\xi,J,\phi)$ is the same as $V^{(0)}(\psi,J,\phi)$, once $\psi$ is re-expressed in terms of the new variables $(\varphi,\xi)$, as in
Lemma \ref{lemma:chi}.

\begin{remark} Note that, because of $\chi(\cdot)$, the sums
    \eqref{taupm} defining $\tau^\pm_{\o,\ell}(x)$ are restricted to
    momenta $k\in\mathcal B_\o^{(2)}$ where $W_\o$ is indeed
    invertible. Note also that, from the smoothness of
    $\hat\tau^\pm_{\omega,\ell}(k)$ it follows that
    $\tau^\pm_{\omega,\ell}(x)$ decays to zero in a
    stretched-exponential way, similar to \eqref{eq:502}. That is,
    $\psi$ is essentially a local function of $\varphi,\xi$.  As a
    consequence, the kernels of $\widetilde V^{(0)}$ satisfy qualitatively
    the same bounds as those of $V^{(0)}$. 
\end{remark}
Since $W_\o(\cdot)$ is  smooth and invertible in the support of $\chi_\o$, we see from \eqref{eq:508} that the propagator of
the variables $\{\xi_{x}(\o)\}$ decays as
\[\|g^W(x,y)\| \leq C e^{-\kappa
  \sqrt{|x-y|}}\]
uniformly in $L$, a 
 behavior  analogous
to \eqref{eq:502}.
For this reason, we call the variables $\{\xi_{x}(\o)\}$
 {\it massive}.
 On the other hand, we call {\it critical} the remaining
$\{\varphi_{x}(\o)\}$ variables.

The integration of the massive
fields $\xi$, which is performed in a way completely analogous to the
one of $\psi^{(1)}$ in \eqref{eq:5.9}, produces an expression for the generating
functional in terms of a Grassmann integral involving only the
critical fields $\varphi$:
\begin{equation}\label{eq:509}
e^{\mathcal{W}_{L}^{(\th)}(A,\phi)}=e^{L^2 E^{(-1)}+S^{(-1)}(J,\phi)}
\int  D\varphi e^{-L^{-2} \sum_{\omega}\sum_{k \in \mathcal{B}_\omega} \hat \varphi^+_{k} (\chi_\o(k))^{-1}{\bf T_\o}(k) \hat \varphi^-_{k}} e^{V^{(-1)}(\varphi,J,\phi)}
\end{equation}
where 
\begin{multline}\label{eq:5.24}
L^2(E^{(-1)}-{\bf t}^{(0)}-E^{(0)})+S^{(-1)}(J,\phi)-S^{(0)}(J,\phi)+V^{(-1)}(\varphi,J,\phi)\\=\log \int P_W(D\xi)e^{\widetilde V^{(0)}(\varphi,\xi,J,\phi)}
\end{multline}
and $V^{(-1)}, S^{(-1)}$ are fixed in such a way that $V^{(-1)}(0,J,\phi)=S^{(-1)}(0,0)=0$. The effective potential $V^{(-1)}$ can be represented in a way similar to \eqref{eq:503}, namely
\begin{equation}
V^{(-1)}(\varphi,J,\phi)=\sum_{\substack{n>0 \\ m,q \geq 0 \\ n+q  \in 2\mathbb{N}}}  \sum^*_{\substack{\underline{x},\underline{y},\un{z} \\ \un \ell',\un \ell'', \un \omega \\ \un s, \un \sigma,\un \sigma'}} \varphi^{\un \sigma}_{\un{x}}(\un{\o}) J_{\un y,\un \ell', \un s} \phi^{\un \sigma'}_{\un z,\un \ell''} W^{(-1)}_{n,m,q;\boldsymbol{a}}
(\underline{x},\underline{y},\un{z};\un{\o})\label{eq:reprV0-1}
\end{equation}
where $\un \omega \in \{-1,+1\}^n$ and $\varphi^{\un \sigma}_{\un{x}}(\un{\o}):=\prod_{i=1}^n \varphi^{\sigma_i}_{x_i}(\o_i)$, while the other symbols and labels have the same meaning as in \eqref{eq:503}. 
In virtue of the decay properties of the propagator $g_{\o,\o'}^W$, the kernels $W^{(-1)}_{n,m,q;\boldsymbol{a}}(\underline{x},\underline{y},\un{z};\un{\o})$ of $V^{(-1)}(\varphi,J,\phi)$ satisfy the same bounds as \eqref{kernorm}. 

\subsection{Reduction to the setting of \cite{GMT20}}
\label{sec:5.3}
We are left with the integral of the critical variables, which we want to
perform in a way analogous to that discussed in \cite[Section
  6]{GMT20}. In order to get to a point where we can literally apply the results of \cite{GMT20}, a couple of extra steps are needed. First, in order to take into account the fact
that, in general, the interaction has the effect of changing the
location of the singularity in momentum space of the propagator of
$\varphi$, as well as the value of the residues at the singularity, we
find it convenient to rewrite the `Grassmann action' in \eqref{eq:509},
\[-L^{-2} \sum_{\omega}\sum_{k \in \mathcal{B}_\omega} \hat \varphi^+_{k} (\chi_\o(k))^{-1}{\bf  T_\o}(k) \hat \varphi^-_{k}+V^{(-1)}(\varphi,J,\phi),\] in the form of a reference quadratic part, with the `right' singularity structure, 
plus a remainder, whose specific value will be fixed a posteriori via a fixed-point argument. More precisely, we proceed as described in \cite[Section 6.1]{GMT20}: we introduce
\begin{equation}
N(\varphi)=L^{-2} \sum_{\omega} \sum_{k \in \mathcal{B}_\omega} \hat\varphi^+_k(-{\bf T_\omega}(p^\omega)+a_\omega(k_1-p^\omega_1)+b_\omega(k_2-p^\omega_2))\hat\varphi^-_k
\end{equation}
where $p^\omega,a_\omega,b_\omega$ will be fixed a posteriori, and are
assumed to satisfy 
\begin{equation}\label{pvicini}|p^\o-p^\o_0| \ll1\end{equation} for $\lambda$ small, $p^+=-p^-$ and
$\overline{a_+}=-a_-, \overline{b_+}=-b_-$. Define
also \begin{equation} C_\o(k):={
    \bf T_\o}(k)-\chi_\o(k)\Big({\bf
  T_\o}(p^\omega)-a_\o(k_1-p^\omega_1)-b_\o(k_2-p^\o_2)
  \Big)
\end{equation}
and note that it satisfies $C_\omega(p^\o)=0$,  \begin{equation}\label{eq:512}
\partial_{k_1}C_\o(p^\omega)=\partial_{k_1}{\bf T_\o}(p^\omega)+a_\omega=:\alpha_\omega, \qquad \partial_{k_2}C_\o(p^\omega)=\partial_{k_2}{\bf T_\o}(p^\omega)+b_\omega=:\beta_\omega,
\end{equation} 
as well as the symmetry $C_{-\o}(-k)=\overline{C_\o(k)}$. 

Let us introduce the matrix $\mathcal{M}$ (the same as in
\cite[Eq.(4.1)]{GMT20}) given by
\begin{equation}
\mathcal{M}=\frac{1}{\sqrt{\Delta}}\begin{pmatrix}
\beta^1 & \beta^2 \\
-\alpha^1 & -\alpha^2
\end{pmatrix}
\end{equation}
where $\alpha^1$ and $\alpha^2$ (resp. $\beta^1$ and $\beta^2$) are,
respectively, the real and imaginary part of $\alpha_+$
(resp. $\beta_+$), see \eqref{eq:512}, and
$\Delta:=\alpha^1\beta^2-\alpha^2\beta^2$ is a positive real number,
in agreement with \eqref{signab}: note, in fact, that at $\lambda=0$
the sign of $\Delta$ is the same as the sign of
$\text{Im}(\beta_+/\alpha_+)$.
At this point, we can finally fix the cut-off functions $\chi_\o$ of  Definition \ref{def:chi} as follows:
\[
\chi_\o(k):=\chi(|\mathcal M^{-1}(k-p^\o)|)
\]
where $\chi:\mathbb R\mapsto [0,1]$ is a compactly supported  function in the Gevrey class of order $2$.
It is immediate to verify that $\chi$ can be chosen so that that properties (i)-(ii) of Definition \ref{def:chi} are verified.

Given this, we
rewrite \eqref{eq:509} as
\begin{equation}\label{eq:510}
e^{\mathcal{W}_{L}^{(\th)}(A,\phi)}=e^{L^2 E^{(-1)}+S^{(-1)}(J,\phi)} \int  D\varphi e^{-L^{-2} \sum_{\omega}\sum_{k \in \mathcal{B}_\omega} \hat \varphi^+_{k} (\chi_\o(k))^{-1} C_\o(k) \hat \varphi^-_{k}} e^{N(\varphi)+V^{(-1)}(\varphi,J,\phi)}.
\end{equation}

In the above integration, the momenta closest to the zeros of $C_\o$ (i.e., close to $p^\o$)
play a special role and have to be treated at the end of the
multiscale procedure, as discussed in \cite[Section 6.5]{GMT20}. For a
given $\theta \in \{-1,+1\}^2$, denote by $k^\pm_\theta \in
\mathcal{B}_\o$ the closest momenta to $p^\pm$ respectively (with the same remark as in footnote \ref{foot:varii} in case of several possible choices)
and
note that they  satisfy $k_\theta^+=-k_\theta^-$. Next we define
$\hat \Phi_\omega^\pm:=\varphi_{k_\theta^\o}^\pm$,
$\Phi_{\omega,x}^\pm:=L^{-2}e^{\pm i k_\theta^\o x}\hat
\Phi_\omega^\pm$ and
$\mathcal{P}'(\theta):=\mathcal{P}(\theta)\setminus\{k_\theta^\pm\}$. Since
$C_\omega$ does not vanish on $\mathcal{P}'(\theta)$, we can rewrite
\eqref{eq:510} as
\begin{multline}\label{eq:511}
e^{\mathcal{W}_{L}^{(\th)}(A,\phi)}=e^{L^2 {\bf E}^{(-1)}+S^{(-1)}(J,\phi)}\\\times  \int D\Phi e^{-L^{-2}\sum_{\omega}\Phi^+_\o C_\omega(k_\theta^\o)\Phi^-_\o} \int  \tilde P_{(\leq -1)}(D\varphi) e^{N(\varphi,\Phi)+V^{(-1)}(\varphi,\Phi,J,\phi)}
\end{multline}
where, letting $\mathcal{B}_\o':=\mathcal{B}_\o \cap \mathcal{P}'(\theta)$, \[{\bf E}^{(-1)}=E^{(-1)}+L^{-2}\sum_{\omega}\sum_{k \in \mathcal{B}_\o' }(\log C_\omega(k)-\log\chi_\o(k)).\] Moreover, 
$D\Phi:=L^4D\hat\Phi_+\, D\hat \Phi_-$ and $\tilde P_{(\leq -1)}(D\varphi)$ is the normalized Grassmann Gaussian integration with propagator 
\begin{equation}
\int \tilde P_{(\le -1)}(D\varphi) \varphi_{\omega,x}^-\varphi^+_{\omega',y}=
\delta_{\o,\o'}\frac{1}{L^2}\sum_{k \in \mathcal{B}_\o'}e^{-ik(x-y)}\chi_\o(k)(C_\o(k))^{-1}.\label{gle0}
\end{equation}

Finally, we remark that since the momenta $k$ in \eqref{gle0} are
  close to $p^\o$, the propagator \eqref{gle0} has an oscillating prefactor $e^{-ip^\o(x-y)}$ that it is convenient to extract. To this end, we define
  {\it quasi-particle} fields $\varphi_{x,\omega}^{\pm,(\le-1)}$ via \[
  \label{eq:qp}\varphi^\pm_x(\omega)=:e^{\pm i p^\omega x} \varphi^{\pm,(\leq -1)}_{x,\omega}.\] Note that the propagator of the quasi-particle fields equals
  \begin{multline}
    \label{eq:bfg}
\int P_{(\leq -1)}(d\varphi^{\leq -1}) \varphi^{-,(\leq -1)}_{x,\omega}
\varphi^{+,(\leq -1)}_{y,\omega'}=\delta_{\o,\o'}g^{(\leq
  -1)}_\omega(x,y)\\ g^{(\leq
  -1)}_\omega(x,y):= \frac{1}{L^2}\sum_{k\in\mathcal
  P'_\o(\theta)}\frac{e^{-ik(x-y)}\chi(k+p^\omega-p^\omega_0)\chi(|\mathcal
  M^{-1}k|)}{C_\omega(k+p^\omega)},
  \end{multline}
where $\mathcal{P}_\omega'(\theta)=\{k: k+p^\omega \in
\mathcal{P}'(\theta)\}$. Of course, the r.h.s. of \eqref{eq:bfg} is
just the r.h.s. of \eqref{gle0} multiplied by $e^{i p^\o(x-y)}$.  We
now rewrite \eqref{eq:511} as
\begin{multline}
  \label{eq:534}
 e^{\mathcal{W}_{L}^{(\th)}(A,\phi)} \\ = e^{L^2 {\bf E}^{(-1)}+S^{(-1)}(J,\phi)}  \int D\Phi e^{-L^{-2}\sum_{\omega}\Phi^+_\o C_\omega(k_\theta^\o)\Phi^-_\o} \int  P_{(\leq -1)}(D\varphi^{(\leq -1)}) e^{\mathcal V^{(-1)}(\varphi^{(\leq -1)},\Phi,J,\phi)},
\end{multline}
where 
\begin{equation}
\mathcal V^{(-1)}(\varphi,\Phi,J,\phi):=N(\Phi,\varphi)+V^{(-1)}(\Phi,\varphi,J,\phi),
\end{equation}
and in the r.h.s. it is meant that the $\varphi$ variables are expressed in terms of the quasi-particle fields as in
\eqref{eq:qp}. That is, we have simply re-expressed $V^{(-1)}$ in terms
of the quasi-particle fields and we included the counter-terms in the definition of effective potential.
After this rewriting, we find that the following representation holds for $\mathcal V^{(-1)}$:
\begin{equation}
\mathcal V^{(-1)}(\varphi,J,\phi)=\sum_{\substack{n>0 \\ m,q \geq 0 \\ n+q  \in 2\mathbb{N}}}  \sum^*_{\substack{\underline{x},\underline{y},\un{z} \\ \un \ell',\un \ell'', \un \omega \\ \un s, \un \sigma,\un \sigma'}} \varphi^{\un \sigma}_{\un{x}, \un{\o}} J_{\un y,\un \ell', \un s} \phi^{\un \sigma'}_{\un z,\un \ell''} \mathcal{W}^{(-1)}_{n,m,q;\underline\omega,\boldsymbol{a}}(\underline{x},\underline{y},\un{z}),\label{eq:reprV-1}
\end{equation}
with kernels $\mathcal{W}^{(-1)}_{n,m,q;\underline\omega,\boldsymbol{a}}$ satisfying the same estimates as in \eqref{kernorm}. 
The kernels $\mathcal{W}^{(-1)}_{n,m,q;\underline\omega,\boldsymbol{a}}$ are the analogues of $W^{(-1)}_{n,m;\underline\omega,\underline r}$ in \cite[Eq.(6.24)]{GMT20} and satisfy the same properties spelled in \cite[Eq.(6.25)]{GMT20} and following lines. Here the labels $\boldsymbol{a}$ denote the collection of labels $(\un\sigma,\un{\ell'},\un s,\un{\ell''},\un\sigma')$. 

 At this point, we have reduced precisely to the 
  fermionic model studied in \cite[Sec. 6]{GMT20}.

\subsection{Infrared integration and conclusion of the proof of Proposition \ref{asimptrelprpgtr}}
\label{sec:5.4}

Once the partition function is re-expressed as in \eqref{eq:534}, we are
in the position of applying the multiscale analysis of \cite[Section
  6]{GMT20}: note in fact that \eqref{eq:534} has exactly the same
form as \cite[Eq. (6.19)]{GMT20} with its second line written as in \cite[Eq.(6.22)]{GMT20}. 
Therefore, at this point, we can
integrate out the massless fluctuation field $\varphi$ via the same
iterative procedure described in \cite[Section 6.2.1]{GMT20} and
following sections. Such a procedure allows us to express the
thermodynamic and correlation functions of the theory in terms of an
appropriate sequence of \textit{effective potentials} $\mathcal{V}^{(h)}$, $h<
0$. The discussion in \cite[Section 6.4]{GMT20} implies that we can fix
$p^\omega, a_\omega, b_\omega$ uniquely as appropriate
\textit{analytic} functions of $\lambda$, for $\lambda$ sufficiently
small  (so that, in particular, \eqref{pvicini} is satisfied), in such a way that the whole sequence of the effective
potentials is well defined for $\lambda$ sufficiently small, their
kernels are analytic in $\lambda$ uniformly in the system size, and
they admit a limit as $L\to\infty$. In particular, the running
coupling constants characterizing the local part of the effective
potentials are analytic functions of $\lambda$ and the associated
critical exponents are analytic functions of $\lambda$, see
\cite[Sects.6.4.5 to 6.4.9]{GMT20}. The existence of the thermodynamic
limit of correlation functions follows from
\cite[Section 6.5]{GMT20}. The proofs of \eqref{hh110}, \eqref{h10ab} and
\eqref{h10a} in Proposition \ref{asimptrelprpgtr} follow from the
discussion in \cite[Section 6.6]{GMT20} (they are the analogues of
\cite[Eqs. (5.1),(5.2) and (5.3) in Proposition 2]{GMT20}) and this,
together with the fact that \eqref{h11} and \eqref{symmrefprop} are
just restatements of \cite[Eq.(4.24)]{GMT20} and
\cite[Eq.(5.8)]{GMT20}, respectively, concludes the proof of
Proposition \ref{asimptrelprpgtr}.

\medskip

A noticeable, even though mostly aesthetic, difference between the statements of Proposition \ref{asimptrelprpgtr}
and \cite[Proposition 2]{GMT20} is in the labeling of the constants $K^{(1)}_{\omega,j,\ell}$ and $K^{(2)}_{\omega,j,\ell}$ in \eqref{hh110}, as compared to those in \cite[Eq.(5.1)]{GMT20}, 
which are called there $\hat K_{\omega,r}$ and $\hat H_{\omega,r}$,
and in the presence of the constants $I_{\omega,\ell,\ell'}$ in \eqref{h10ab}-\eqref{h10a}, which are absent in their analogues in \cite[Eq.(5.2)-(5.3)]{GMT20}. 
This must be traced back to the different labeling of the sites and edges and, correspondingly, of the external fields $\phi$ and $A$, used in this paper, as compared to \cite{GMT20}. 

First of all, in this paper the edges and the external fields of type $A$ are labelled $(x,j,\ell)$, with $(j,\ell)$ playing the same role as the index $r$ in \cite{GMT20}; 
correspondingly, the analogues of the running coupling constants $Y_{h,r,(\omega_1,\omega_2)}$ defined in \cite[Eq.(6.49)]{GMT20} should now be labelled $Y_{h,(j,\ell),(\omega_1,\omega_2)}$; by repeating the discussion in \cite[Section 6.6]{GMT20} leading to \cite[Eq.(6.160)]{GMT20}, it is apparent that the analogues of the 
constants $\hat K_{\omega,r}, \hat H_{\omega,r}$ should now be labelled $(\omega,j,\ell)$, as anticipated. 

Concerning the constants $I_{\omega,\ell,\ell'}$, they come from the local part of the effective potentials in the presence of the external fields $\phi$. After having 
integrated out the massive degrees of freedom, the infrared integration procedure involves at each step a splitting of the effective potential into a sum of its local part $\mathcal L\mathcal V^{(h)}$ and of its `renormalized', or `irrelevant', part $\mathcal R \mathcal V^{(h)}$, as discussed in \cite[Section 6.2.3]{GMT20}. In  \cite[Section 6]{GMT20}, for simplicity, we discussed the infrared integration only in the absence of external $\phi$ fields. In their presence, the definition of localization must be adapted accordingly. When acting on the $\phi$-dependent part of the effective potential, using a notation similar to \cite[Eq.(6.37)]{GMT20}, we let 
\begin{equation}\begin{split}
\mathcal L \Big(\mathcal V^{(h)}(\varphi,J,\phi)-\mathcal V^{(h)}(\varphi,J,0)\Big)&=\sum_{x\in \Lambda}\sum_{\omega,\ell} \Big(\varphi^+_{x,\omega}\phi^-_{x,\ell}e^{ip^\omega\cdot x}\hat{\mathcal W}^{(h),\infty}_{1,0,1;\omega,(+,\ell,-)}(0)\\ & + 
\varphi^-_{x,\omega} \phi^+_{x,\ell} e^{-ip^\omega\cdot x}\hat{\mathcal W}^{(h),\infty}_{1,0,1;\omega,(-,\ell,+)}(0)\Big).\end{split}
\end{equation}
Next, in analogy with \cite[Eq.(6.49)]{GMT20}, we let 
\begin{equation}\label{rccI}
  I^\pm_{h,\omega,\ell}:=\frac1{\sqrt{Z_{h-1}}}\hat W^{(h),\infty}_{1,0,1;\omega,\ell,(\pm,\mp)}(0),\end{equation}
where $Z_h$ is a real, scalar, function of $\lambda$, called the `wave function renormalization', recursively defined as in \cite[Eq.(6.45)]{GMT20}. Eq.\eqref{rccI}
defines the running coupling constant (r.c.c.) associated with the external field $\phi$. Note that such r.c.c. naturally inherit the label $\ell$ from the corresponding label of the 
external field $\phi$. A straightforward generalization of the discussion in \cite[Section 6.4]{GMT20} shows that $I^\pm_{h,\omega,\ell}$ are analytic in $\lambda$ and converge 
as $h\to-\infty$ to finite constants $I^\pm_{-\infty,\omega,\ell}$, which are, again, analytic functions of $\lambda$. Therefore, by repeating the discussion in \cite[Section 6.6]{GMT20}
for $\hat G^{(2)}_{\ell,\ell'}(k+p^\omega)$ and $\hat G^{(2,1)}_{j,\ell_0,\ell,\ell'}(k+p^\omega,p)$, we find that the dominant asymptotic behavior of these correlations as $k,p\to 0$ 
is proportional to 
$I^+_{-\infty,\omega,\ell}I^-_{-\infty,\omega,\ell'}$, times a function that is independent of $\ell,\ell'$. Building upon this, we obtain \eqref{h10ab} and \eqref{h10a}, with 
$I_{\omega,\ell,\ell'}$ proportional to $I^+_{-\infty,\omega,\ell}I^-_{-\infty,\omega,\ell'}$. Additional details are left to the reader. 

\section{Proof of Theorem \ref{th:GFF}}
\label{sec:6}
In order to prove Theorem \ref{th:GFF} we proceed as in
\cite[Section~7.3]{GMTaihp}: using the fact that convergence of the
moments of a random variable $\xi_n$ to those of a Gaussian random
variable $\xi$ implies convergence in law of $\xi_n$ to $\xi$, we
reduce the proof of \eqref{eq:convGFF} to that of the following
identities:
\begin{equation}\label{twoeq} \begin{split}
& \lim_{\epsilon\to 0}\mathbb E_\lambda(h^\epsilon(f);h^{\epsilon}(f))=\frac{\nu(\lambda)}{2\pi^2}\int\, dx\int \, dy\, f(x)\, f(y)\, \mathfrak{R}[\log\phi_+(x-y)].\\
& \lim_{\epsilon\to 0}\mathbb E_\lambda(\underbrace{h^\epsilon(f);\cdots;h^{\epsilon}(f)}_{n\ \text{times}})=0, \qquad n>2
\end{split}
\end{equation}
where the l.h.s. of the second line denotes the $n^{th}$ cumulant of $h^\epsilon(f)$.
The first equation is a straightforward corollary of Theorem \ref{mainthrm}, for additional details see \cite[p.161, proof of (7.26)]{GMTaihp}. For the proof of the second equation we 
need to show that, for any $2n$-ple of distinct points $x_1,\ldots,x_{2n}$, 
\begin{equation}\label{higher} \mathbb E_\lambda(h(\eta_{x_1})-h(\eta_{x_2});\cdots;h(\eta_{x_{2n-1}})-h(\eta_{x_{2n}}))=O((\min_{1\le i<j\le 2n}|x_i-x_j|)^{-\theta}),\end{equation}
for some constant $\theta>0$. In fact, by proceeding as in
\cite[p.162, Proof of (7.27)]{GMTaihp}, Eq.~\eqref{higher} readily
implies the second line of \eqref{twoeq}. In order to prove
\eqref{higher}, we first expand each difference within the expectation
in the left side as in \eqref{heightfunction}, thus getting
\begin{equation} \text{LHS of \eqref{higher}}=\sum_{e_1\in C_{\eta_{x_{1}}\to \eta_{x_{2}}}}\cdots \sum_{e_n\in C_{\eta_{x_{2n-1}}\to \eta_{x_{2n}}}}\sigma_{e_1}\cdots\sigma_{e_n}
\mathbb E_\lambda(\mathds 1_{e_1};\cdots;\mathds 1_{e_n}).\label{uffa}\end{equation}
At a dimensional level, the truncated
$n$-point correlation in the right hand side decays like
$d^{-n(1+O(\lambda))}$, where
$d$ is the minimal pairwise distance among the edges
$e_1,\ldots,e_n$; therefore, the result of the
$n$-fold summation in \eqref{uffa} is potentially unbounded as
$\max_{i<j}|x_i-x_j|\to\infty$. In order to show that this is not the
case, and actually the result of the
$n$-fold summation is bounded as in the right hand side of
\eqref{higher}, we need to exhibit appropriate cancellations. Once
more, we use the comparison of the dimer lattice model with the
infrared reference model, which allows us to re-express the multi-point
truncated dimer correlation $\mathbb E_\lambda(\mathds
1_{e_1};\cdots;\mathds
1_{e_n})$ as a dominant term, which is the multi-point analogue of
\eqref{rel2p}, plus a remainder, which decays faster at large
distances. More precisely, by using a decomposition analogous to
\cite[Eq.~(7.7)]{GMTaihp} and using the analogue of
\cite[Eq.~(6.90)]{GMTaihp}, if $e_i$ has labels
$(x_i,j_i,\ell_i)$, we rewrite
\begin{equation}\label{epperro} \begin{split}\mathbb E_\lambda(\mathds 1_{e_1};\cdots;\mathds 1_{e_n})&=\sum_{\substack{\omega_1,\ldots,\omega_n=\pm\\ s_1,\ldots,s_n=1,2}} \Big({\prod_{r=1}^nK^{(s_r)}_{\omega_r,j_r,\ell_r}\big(\prod_{r:s_r=2}e^{2ip^{\omega_r}\cdot x_r}}\big)\Big) 
    S^{(s_1,\ldots,s_n)}_{R;\omega_1,\ldots,\omega_n}(x_1,\ldots,x_n)\\
    &+\text{Err}(e_1,\ldots,e_n),\end{split}\end{equation} where
$S^{(s_1,\ldots,s_n)}_{R;\omega_1,\ldots,\omega_n}$ are the multi-point
density-mass correlations of the reference model (defined as in
\cite[Eq.~(6.85)]{GMTaihp} or as the multi-point analogue of
\cite[Eq.~(4.15)-(4.16)]{GMT20}). Moreover, if $D_{\underline
  x}$ is the diameter of $\underline
x=(x_1,\ldots,x_n)$ and if the minimal separation among the elements
of $\underline x$ is larger than $c_0 D_{\underline
  x}$ for some positive constant $c_0$, then, for
$\theta$ equal to, say, $1/2$ (in general,
$\theta$ can be any positive constant smaller than
$1-O(\lambda)$) the remainder term is bounded as
$|\text{Err}(e_1,\ldots,e_n)|\le C_{n,\theta}(c_0)D_{\underline
  x}^{-n-\theta}$. The latter bound is the analogue of
\cite[Eq.~(6.90)]{GMTaihp}.  Moreover\footnote{See the discussion after \cite[Eq.~(6.94)]{GMTaihp} for
references about the properties of $S^{(s_1,\dots,s_n)}_{R;\o_1,\dots,\o_n}$ that are discussed in this paragraph.}, the functions
$S^{(s_1,\ldots,s_n)}_{R;\omega_1,\ldots,\omega_n}$ are non-zero only
if the quasi-particle indices satisfy the constraint $\sum_{i:
  s_i=2}\omega_i=0$ (this is the multi-point generalization of
\cite[Eq.~(6.92)]{GMTaihp}). Finally, and most importantly, if
$s_1=\cdots=s_n=1$, then
\begin{equation}
  S^{(1,\ldots,1)}_{R;\omega_1,\ldots,\omega_n}(x_1,\ldots, x_n)\equiv
  0, \qquad n>2,\label{questa}\end{equation} which is the analogue of
\cite[Eq.(6.94)]{GMTaihp} and is an instance of `bosonization' for
the reference model: in fact, \eqref{questa} can be interpreted by
saying that the $n$-point (with
$n>2$) truncated density correlations  of the reference model (recall Footnote \ref{foot.mass} for the definition of `density' and `mass' observables) are all
identically equal to zero. 

In conclusion, in the right hand side of \eqref{uffa} we can replace
$\mathbb E_\lambda(\mathds 1_{e_1};\cdots;\mathds
1_{e_n})$ by the right hand side of \eqref{epperro}, where the term
with
$s_1=\cdots=s_n=1$ vanishes. Therefore, all the terms we are left with
either involve oscillating factors $\prod_{r:
  s_r=2}e^{2ip^{\omega_r}\cdot
  x_r}$ or the remainder term
$\text{Err}(e_1,\ldots,e_n)$. In both cases, exactly like in the case
$n=2$, the contribution of these terms to the
$n$-fold summation over
$e_1,\dots,e_n$ in \eqref{uffa} is bounded better than the naive dimensional
estimate, and we are led to the bound in \eqref{higher}. For a
detailed discussion of how the estimate of the summation is performed,
we refer the reader to \cite[Section~7.2]{GMTaihp}.

\appendix

\section{An explicit example of non-planar dimer model}\label{app:B}
Here, we work out the Grassmann potential $V$ for the easiest but non-trivial example of non-planar dimer model.
Choose $m=4$ for the cell size and let  the edge weights be invariant by translations by multiples of $m$, so that
$V_x$ in Proposition \ref{prop:1} does not depend on $x$. In this
example we add just one non planar edge per cell, denoted by $e_\lambda$,
connecting the leftmost black site in the second row to the rightmost
white in the same row; it crosses two vertical edges, denoted by
$e_1,e_2$, see Figure \ref{fig:12}. \begin{figure}[H]
\centering
\includegraphics[scale=0.7]{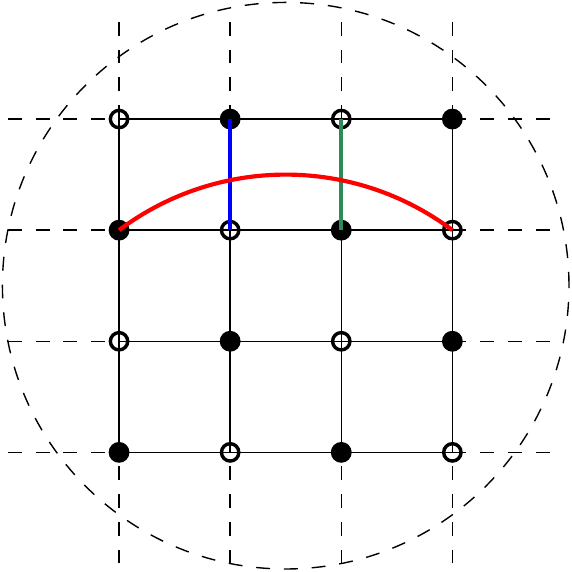}
\caption{A $4\times 4$ cell with the edges $e_\l,e_1,e_2$ colored in red, blue, green, respectively.}
\label{fig:12}
\end{figure}
Let
$\psi(e_\lambda),\psi(e_1),\psi(e_2)$ be the Grassmann monomials defined in \eqref{eq:psie} (we drop the index $\theta$). From
the definition \eqref{eq:Vx}, one can check that the potential
satisfies $V(\psi)=F(\psi)$ and that it is given by
\begin{equation}
V(\psi)=\varepsilon_\emptyset^{\{e_\lambda\}}\psi(e_\lambda)+\varepsilon^{\{e_\l\}}_{\{e_1\}} \psi(e_\l)\psi(e_1)+\varepsilon^{\{e_\l\}}_{\{e_2\}} \psi(e_\l)\psi(e_2)+\varepsilon^{\{e_\l\}}_{\{e_1,e_2\}} \psi(e_\l) \psi(e_1) \psi(e_2).
\end{equation}
The computation of the signs $\varepsilon_S^{J}$  can be easily done starting from \eqref{veryverysimple} and with the help of Fig. \ref{fig:ex}; details are left to the reader.
The final result  is that
\begin{equation}
\varepsilon_\emptyset^{\{e_\lambda\}}=\varepsilon^{\{e_\l\}}_{\{e_1,e_2\}}=1, \varepsilon^{\{e_\l\}}_{\{e_1\}}=\varepsilon^{\{e_\l\}}_{\{e_2\}}=-1.
\end{equation}
\begin{figure}[H]
\centering
\includegraphics[scale=1.1]{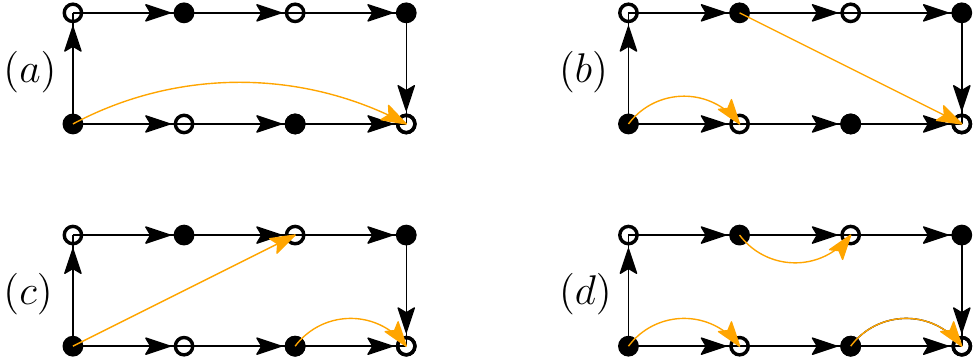}
\caption{The set of edges $E_{J,S}$ with  $J=\{e_\l\},S=\emptyset$ (drawing (a)), $J={\{e_\l\},S=\{e_1\}}$ (drawing (b)), $J={\{e_\l\},S=\{e_2\}}$ (drawing (c)), $J={\{e_\l\},S=\{e_1,e_2\}}$ (drawing (d)), colored in orange. Here the orientation of black edges coincides with that on $G^0_L$ (see Fig. \ref{fig:7}), while that of orange edges is the one described in Lemma \ref{lem:KJS} and in the caption of Figure \ref{fig:5}.}
\label{fig:ex}
\end{figure}

\section*{Acknowledgments}
We wish to thank Beno\^it Laslier for contributing with ideas and discussions in the early stages of this project.
A. G. gratefully acknowledges financial support of the European Research Council (ERC)
under the European Union’s Horizon 2020 research and innovation programme (ERC
CoG UniCoSM, grant agreement n. 724939), and of  MIUR, PRIN 2017 project MaQuMA,
PRIN201719VMAST01.
F. T. gratefully acknowledges financial support of the  Austria Science Fund (FWF), Project Number P 35428-N.


\begin{thebibliography}{99}

\bibitem{AHPS} M. Aizenman, M. Harel, R. Peled, and J. Shapiro, \emph{Depinning in the integer-valued Gaussian
Field and the BKT phase of the 2D Villain model}, arXiv:2110.09498.

\bibitem{Alet} {F. Alet et al, \emph{Interacting Classical Dimers on the Square Lattice}, Phys. Rev. Lett. {\bf 94} (2005), 235702.}

\bibitem{Beffara-Johansson} {V. Beffara, S. Chhita, K. Johansson, \emph{Airy point process at the liquid-gas boundary}, 
Annals of Probability {\bf 46} (2018), 2973--3013.}

\bibitem{Roland1} R. Bauerschmidt, J. Park, P.-F. Rodriguez, \emph{The Discrete Gaussian model, I. Renormalisation group flow at high temperature}, arXiv:2202.02286

\bibitem{Roland2}  R. Bauerschmidt, J. Park, P.-F. Rodriguez, \emph{The Discrete Gaussian model, II. Infinite-volume scaling limit at high temperature}, arXiv:2202.02287

\bibitem{BFM09a} G. Benfatto, P. Falco, V. Mastropietro: \emph{Massless Sine-Gordon and 
massive Thirring models: proof of Coleman's equivalence} Commun. Math. Phys. {\bf 285} (2009) 713--762.

\bibitem{BFM09b} G. Benfatto, P. Falco, V. Mastropietro: \emph{Extended scaling relations for planar lattice models}, Commun.
Math. Phys. {\bf 292} (2009), 569--605.

\bibitem{BFM14}  G. Benfatto, P. Falco, V. Mastropietro: \emph{Universality of One-Dimensional Fermi Systems, I. Response Functions and Critical Exponents}, Commun. Math. Phys. {\bf 330} (2014), pp.153--215; and \emph{II. The Luttinger Liquid Structure}, ibid., pp. 217--282.

\bibitem{BM02} G. Benfatto, V. Mastropietro: \emph{On the density-density critical indices in interacting Fermi systems}, Commun.
Math. Phys. {\bf 231} (2002), 97--134.

\bibitem{BM04} G. Benfatto, V. Mastropietro: \emph{Ward identities and vanishing of the Beta function for $d=1$ interacting
Fermi systems}, J. Stat. Phys. {\bf 115} (2004), 143--184.

\bibitem{BM05} G. Benfatto, V. Mastropietro: \emph{Ward identities and chiral anomaly in the Luttinger liquid}, Commun. Math. Phys. {\bf 258} (2005), 609--655.

\bibitem{BM10} G. Benfatto, V. Mastropietro: \emph{Universality relations in non-solvable quantum spin chains}, J. Stat. Phys. {\bf 138} (2010), 1084--1108.

\bibitem{BM11} G. Benfatto, V. Mastropietro: \emph{Drude weight in non solvable quantum spin chains}, J. Stat. Phys. {\bf 143} (2011), 251--260.

\bibitem{CPST} N. Chandgotia, R. Peled, S. Sheffield, and M. Tassy, \emph{Delocalization of uniform graph homomorphisms from $\mathbb Z^2$  to $\mathbb Z$}, Commun. Math. Phys. {\bf 387} (2021), no. 2, 621–647.

\bibitem{Dobrushin} R. L. Dobrushin, \emph{The Gibbs state that
    describes the coexistence of phases for a three-dimensional Ising
    model}, Teor.  Verojatnost. i Primenen., {\bf 17} (1972), 619–639.
  
\bibitem{DC1} H. Duminil-Copin, M. Harel, B. Laslier, A. Raoufi, and G. Ray, \emph{Logarithmic variance for
  the height function of square-ice}, Comm. Math. Phys. {\bf 396}, 867-902 (2022).

\bibitem{DC2} H. Duminil-Copin, A. Karrila, I. Manolescu, and M. Oulamara, \emph{Delocalization of the height
  function of the six-vertex model}, arXiv:2012.13750

\bibitem{Fa1} P. Falco, \emph{Kosterlitz-Thouless transition line for the
  two-dimensional Coulomb gas}, Commun. Math. Phys. {\bf 312}(2):559-609, 2012

\bibitem{Fa2}P. Falco, \emph{Critical exponents of the two dimensional
  Coulomb gas at the Berezinskii-Kosterlitz-Thouless transition}, arXiv:1311.2237
  
\bibitem{FS} J. Fröhlich and T. Spencer, \emph{The Kosterlitz-Thouless transition in two-dimensional abelian spin
systems and the Coulomb gas}, Comm. Math. Phys. {\bf 81} (1981), no. 4, 527–602.

\bibitem{Galluccio} A. Galluccio, M. Loebl, \emph{On the Theory of Pfaffian Orientations.
  I. Perfect Matchings and Permanents}, Elect. J. Combinatorics {\bf 6} (1999), \#R6

\bibitem{GM01} G. Gentile, V. Mastropietro, \emph{Renormalization group for one-dimensional
fermions. A review on mathematical results}, Phys. Rep. {\bf 352}, 273-438 (2001).

\bibitem{Gia} T. Giamarchi: \emph{Quantum Physics in One Dimension}, Oxford University Press, Oxford (2004).

\bibitem{GGM12} A. Giuliani, R.L. Greenblatt, V. Mastropietro, \emph{The scaling limit of the energy
correlations in non integrable Ising models}, Jour. Math. Phys. {\bf 53}, 095214 (2012).

\bibitem{GMTaihp} A. Giuliani, V. Mastropietro, F. L. Toninelli, \emph{Height fluctuations in interacting dimers}, 
Ann. Inst. H. Poincar\'e Probab. Statist. 53(1): 98-168 (2017).
  
\bibitem{GMTHaldane} A. Giuliani, V. Mastropietro, F. L. Toninelli, \emph{Haldane relations for interacting dimers}, J. Stat. Mech. (2017), 10.1088/1742-5468/aa5d1f. 

\bibitem{frattaglie} A. Giuliani, F. L. Toninelli, \emph{Non-integrable dimer models: universality
and scaling relations}, J. Math. Phys. {\bf 60}, 103301 (2019).

\bibitem{GMT20} A. Giuliani, V. Mastropietro, F. L. Toninelli, \emph{Non integrable Dimers: Universal Fluctuations of Tilted Height Profiles},  Commun. Math. Phys. {\bf 377} (2020), 1883-1959.

\bibitem{Gorin} V. Gorin,  \emph{Lectures on random lozenge tilings}, Vol. 193, Cambridge University Press, 2021.
  
\bibitem{Haber} H. E. Haber, \emph{Notes on antisymmetric matrices and the pfaffian}, unpublished, available at http://scipp.ucsc.edu/~haber/webpage/pfaffian2.pdf

\bibitem{HP} {O. J. Heilmann, E. Praestgaard, \emph{Crystalline ordering in lattice models of hard rods with nearest neighbor attraction}, Chem. Phys. {\bf 24} (1977), 119--123.}

\bibitem{Johansson} {K. Johansson, \emph{The arctic circle boundary and the Airy process}, Annals
of Probability {\bf 33} (2005), 1–30.}

\bibitem{Kasteleyn}P. W. Kasteleyn, \emph{The statistics of dimers on a lattice: I. The number of dimer
  arrangements on a quadratic lattice}, Physica {\bf 27}, 1209-1225 (1961).
  
\bibitem{Kasteleyn2}P. W. Kasteleyn, \emph{Dimer statistics and phase transitions}, J.  Math. Phys. {\bf 4} (1963), no. 2, 287-293.

\bibitem{Kast67} P. W. Kasteleyn,  \emph{Graph theory and crystal physics}, Graph theory and theoretical physics (1967), 43-110
  
\bibitem{Kenyonlocal} R. Kenyon, \emph{Local statistics of lattice dimers}, Ann. Inst. H. Poincaré, Prob Stat.
  {\bf 33} (1997), 591-618.
  
\bibitem{Kenyon_notes} R. Kenyon, \emph{Lectures on dimers}, Park City Math Institute Lectures, available at
  arXiv:0910.3129.
  
\bibitem{KOS} R. Kenyon, A. Okounkov, S. Sheffield, \emph{Dimers and Amoebae}, 
Annals Math. {\bf 163}, 1019-1056 (2006).

\bibitem{Lammers2} P. Lammers, \emph{ Height function delocalisation on cubic planar graphs}, Probab. Theory Related
  Fields {\bf 182} (1-2):531–550, 2022.
  
\bibitem{Lammers1}
  P. Lammers and S. Ott, \emph{Delocalisation and absolute-value-FKG in the solid-on-solid model}, Probab. Theory Rel. Fields, to appear,
  arXiv:2101.05139

\bibitem{LiWu} E. H. Lieb, F. Y. Wu: \emph{Absence of Mott transition in the 1D Hubbard model}, Phys. Rev. Lett. {\bf 20} (1968), 1445--1449.

\bibitem{Lis} M. Lis, \emph{On Delocalization in the Six-Vertex Model}, Commun.  Math. Phys. {\bf 383} (2021),  1181–1205.
  
\bibitem{TemperleyFisher} N. V. Temperley, M. E. Fisher, \emph{Dimer problem in statistical mechanics-an exact
result}, Philos. Mag. {\bf 6}, 1061-1063 (1961).

\bibitem{Tesler} G. Tesler,  \emph{Matchings in graphs on non-orientable surfaces}, Journal of Combinatorial Theory, Series B {\bf 78} (2000), no. 2, 198-231.

\bibitem{Tsv} A. M. Tsvelik: \emph{Quantum Field Theory of Condensed Matter Physics}, Cambridge University Press, Cambridge (2007).

\bibitem{Wu} W. Wu, \emph{A central limit theorem for square ice}, arXiv:2206.12058


\end{thebibliography}
\end{document}